\documentclass[11pt,reqno,oneside]{amsart}

\usepackage{tikz}
\usepackage{graphicx, subfigure,}
\usepackage{amssymb}
\usepackage{epstopdf}
\usepackage{amssymb}
\usepackage[a4paper, total={6in, 9.1in}]{geometry}

\usepackage{bm}
\usepackage{amsmath,amsfonts,amsthm,mathrsfs,amssymb,cite}

\usepackage{graphics}
\usepackage{framed}

\usepackage{enumerate}
\usepackage{hyperref}
\usepackage{xcolor}
\hypersetup{
	colorlinks,
	linkcolor={blue!80!black},
	urlcolor={blue!80!black},
	citecolor={blue!80!black}
}
\usepackage[capitalize,noabbrev]{cleveref}

\usepackage{xcolor}

\numberwithin{equation}{section}

\newtheorem{theorem}{Theorem}[section]
\newtheorem{lemma}[theorem]{Lemma}

\theoremstyle{definition}

\theoremstyle{remark}
\newtheorem{remark}[theorem]{Remark}

\numberwithin{equation}{section}

\newcounter{saveeqn}
\newcommand{\ba}{\begin{array}}
	\newcommand{\ea}{\end{array}}

\newcommand{\bea}{\begin{eqnarray*}}
	\newcommand{\eea}{\end{eqnarray*}}
\newcommand{\bean}{\begin{eqnarray}}
	\newcommand{\eean}{\end{eqnarray}}

\def\dist{\mathrm{dist}}

\allowdisplaybreaks[4]

\usepackage{xstring} % 

\newcommand{\AuthorInfo}[4]{%
	\textsc{#1}%
	\IfStrEq{#4}{true}{$^{*}$}{}\\
	#2\\
	\textit{E-mail:} \texttt{#3}%
	\IfStrEq{#4}{true}{\\\textit{$^{*}$Corresponding author.}}{}
	\par\vspace{0.5em}
}

\usepackage{authblk}

\title[Construction of Customized Solutions for Schr\"odinger Equations]
{Construction of Solutions with Extraordinary Gradient Amplification and Localization for Schr\"odinger Equations }

\date{} % Activate to display a given date or no date (if empty),
\begin{document}
	
	\maketitle
	
	\begin{center}
		\bigskip
		\footnotesize

		\AuthorInfo{Huaian Diao}{School of Mathematics and Key Laboratory of Symbolic Computation and Knowledge Engineering of Ministry of Education, Jilin University, Changchun 130012, China}{diao@jlu.edu.cn, hadiao@gmail.com}{false}
		
		\AuthorInfo{Xieling Fan}{School of Mathematics, Central South University, Changsha, China and Department of Mathematics, City University of Hong Kong, Kowloon, Hong Kong SAR, China}{fanxieling@outlook.com, xielinfan2-c@my.cityu.edu.hk}{false}
		
		\AuthorInfo{Hongyu Liu}{Department of Mathematics, City University of Hong Kong, Kowloon, Hong Kong SAR, China}{hongyu.liuip@gmail.com, hongyliu@cityu.edu.hk}{false}
	\end{center}

	\normalsize

	\begin{abstract}
This paper constructs solutions to linear and nonlinear Schr\"odinger-type equations in two and three spatial dimensions that exhibit prescribed, extraordinary gradient amplification and localization. For any finite time interval $[0,T]$, any prescribed collection of $n\in\mathbb{N}$ distinct points on $\partial D$, where $D$ is the compact support of the anisotropic coefficients, lower-order terms, or nonlinearities, and any amplitude threshold $\mathcal{M}>0$, we show that one can design smooth initial and/or boundary data such that the spatial gradients of the resulting solutions exceed $\mathcal{M}$ in neighborhoods of these points outside $D$ for almost every $t\in[0,T]$. Moreover, the ratio between the local $C^{1,\frac12}$-norm of the solution near each prescribed point outside $D$ and the $C^{1,\frac12}$-norm inside $D$ is bounded from below by $\mathcal{M}/2$ for almost every $t\in[0,T]$.
	
	We further prove that the spatial measure of the regions where the gradient magnitude exceeds $\mathcal{M}$ tends to zero as $\mathcal{M}\to\infty$, demonstrating that the amplification phenomenon is highly localized.  This effect arises from the structure of the
	Schr\"odinger-type equation combined with carefully designed input profiles. From a physical perspective, the results provide a deterministic analogue of localization phenomena observed in quantum scattering and Anderson localization. In addition, the observed trade-off between extreme spatial localization and large gradient amplification is fully consistent with the spirit of the Heisenberg uncertainty principle: while the latter is traditionally formulated in a global $L^2$ space--frequency framework, our results offer a complementary deterministic manifestation at the level of localized spatial gradients in Schr\"odinger dynamics.

		\medskip

		
		81Q10
		
	\end{abstract}

    \section{Introduction}
	
\subsection{Background and discussion}

The Schr\"odinger equation is one of the fundamental dispersive models in mathematical physics. In its time-dependent form, it describes the evolution of complex wave functions and provides the basic framework for the analysis of non-relativistic quantum dynamics. From the viewpoint of partial differential equations, its dispersive structure gives rise to a rich analytic theory, including decay estimates, Strichartz estimates, scattering theory, and delicate regularity phenomena; see, for instance, \cite{Cazenave2003Semilinear, Tao2006NDE}. In many physically relevant settings, one must also incorporate nonlinear self-interaction effects, which leads to nonlinear Schr\"odinger-type equations. A canonical model is the nonlinear Schr\"odinger equation
\begin{equation}\label{eq:NLS}
	i\,\partial_{t} u + \Delta u + \kappa\,|u|^{p-1} u = 0,
	\qquad (x,t) \in \mathbb{R}^{d} \times \mathbb{R},
\end{equation}
where $\kappa \in \mathbb{R}$ and $p > 1$. The focusing and defocusing regimes exhibit markedly different dynamics, and the equation arises naturally in nonlinear optics, Bose--Einstein condensation, and water wave theory; see \cite{SulemSulemNLS, Cazenave2003Semilinear, Tao2006NDE}.

The mathematical theory of Schr\"odinger equations has developed along several major directions, among which well-posedness, scattering, and singularity formation are particularly prominent. In the focusing regime, finite-time blow-up is a central phenomenon. Classical results show that sufficiently localized smooth initial data may generate singular behavior in supercritical settings \cite{Glassey1977Blowup}, while in critical regimes one may even prescribe blow-up at finitely many spatial points \cite{Merle1990kBlowup, Fan2017LogLogMultiPoint}. On the other hand, localization phenomena also play an important role in the spectral and dynamical theory of Schr\"odinger operators. A representative example is Anderson localization \cite{Anderson1958}, where random or quasi-periodic structures may force eigenfunctions to concentrate strongly near certain spatial regions; see, e.g., \cite{Cycon1987, ReedSimon3, Taylor1972, Yafaev1992}. These two directions, singularity formation and localization, reflect in distinct manners the capacity of Schr\"odinger dynamics to generate highly nontrivial spatial concentration.

The purpose of the present work is to establish a distinct concentration mechanism for a broad class of linear and nonlinear Schr\"odinger-type equations. More precisely, we show that, over any prescribed finite time interval $[0,T]$, one can construct smooth initial and/or boundary data such that the corresponding solution remains locally $H^{3}$-regular, yet exhibits extraordinarily large spatial gradients in arbitrarily small neighborhoods of any prescribed finite collection of boundary points. This phenomenon is fundamentally different from classical blow-up: no loss of regularity occurs, and the large-gradient behavior is both controlled and customizable. It is also different in nature from Anderson localization, which is induced by randomness or spectral disorder. Here the localization is deterministic and constructive, produced through carefully designed inputs and realized within globally regular Schr\"odinger evolutions.

A further feature of our construction is that the region of amplification becomes quantitatively small as the prescribed gradient threshold increases. More precisely, if $\mathcal M$ denotes the target size of the gradient, then the measure of the set on which the gradient exceeds $\mathcal M$ decays like $O(\mathcal M^{-d})$; see Theorems~\ref{thm:main} and~\ref{thm:main2}. Thus, increasingly large gradients are achieved on increasingly small spatial sets. In this sense, our results reveal a quantitative concentration mechanism that is compatible with the dispersive character of Schr\"odinger dynamics

	\subsection{Mathematical setup and summary of main results}

We focus on the mathematical models for both linear and nonlinear Schr\"{o}dinger equations in two and three spatial dimensions in this paper, where the mechanism for these mathematical models will be elaborated in more detail later. Consider that $\Omega$ is a bounded Lipschitz domain in $\mathbb{R}^d$, $d=2, 3$, and $0 < T < \infty$. Furthermore, assume that $D$ is a bounded Lipschitz domain such that $D \Subset \Omega$, and the region $\Omega \setminus \overline{D}$ is connected. Throughout this paper, $\overline{\cdot}$ denotes the complex conjugate and $\top$ denotes the transpose. We define the partial differentiation operator $\mathcal{L}_1$ as follows:
\begin{equation}\notag
    \mathcal{L}_1 = 
\nabla \cdot (\mathscr{A}_1 
\nabla ),
\end{equation}
where the complex-valued coefficient matrix $\mathscr{A}_1 = (a_1^{ij}(\mathbf{x}))_{i,j=1}^d$ satisfies the Hermitian property $\overline{\mathscr{A}}_1^{\top}(\mathbf{x}) = \mathscr{A}_1(\mathbf{x}), ~\forall \mathbf{x} \in \Omega$, and belongs to the regularity class $\mathscr{A}_1 \in W^{2,\infty}(\Omega) \cap C^{1,1}(\overline{\Omega})$. The operator is assumed to be strongly elliptic in the sense that there exists a constant $\theta > 0$ such that:
\begin{equation}\notag
    \overline{\xi}^{\top} \mathscr{A}_1(\mathbf{x}) \xi \geq \theta |\xi|^{2}
\end{equation}
for a.e. $\mathbf{x} \in \Omega$ and for all $\xi \in \mathbb{C}^d$. Furthermore, assume that $\mathscr{A}_1- \mathbf{I}$ has compact support in $D$, namely,
		\begin{equation*}\label{compact support}
			\operatorname{supp}(\mathscr{A}_1 - \mathbf{I}) \subset D,
		\end{equation*}
        where $\mathbf{I}$ is the $d$-dimensional  identity matrix. 
        Under these geometric and structural assumptions, let the unknown function $u: \Omega \times [0,T] \rightarrow \mathbb{C}$ satisfy the following homogeneous initial-boundary value problem for the linear Schr\"{o}dinger equation:
\begin{equation}\label{1.1}
    \begin{cases}
        \mathrm{i}\partial_t u + \mathcal{L}_1 u = 0 & \text{in } \Omega \times (0,T], \\
        u(\mathbf{x},0) = \phi(\mathbf{x}) & \text{in } \Omega, \\
        u(\mathbf{x},t) = \psi(\mathbf{x},t) & \text{on } \partial\Omega \times [0,T],
    \end{cases}
\end{equation}
where $\mathrm{i}$ denotes $\sqrt{-1}$. Here the initial data $\phi\in C^{\infty}{(\Omega)}$, and boundary data $\psi \in C^\infty([0, T];\\ C(\partial \Omega))$ will be specified in Section \ref{section 3}.

Similarly, for the Cauchy problem of a nonlinear Schr\"{o}dinger-type equation in the full space $\mathbb{R}^d$, we consider a bounded Lipschitz domain $D \subset \mathbb{R}^d$ with a connected complement. For any given time interval $T \in \mathbb{R}_+$, we assume that the zero-order term $c(\mathbf{x},t)$ satisfies:
\begin{equation}\notag
    c \in L^{\infty}\bigl(0,T; W^{3,\infty}(\mathbb{R}^d)\bigr),
\end{equation}
and its spatial support is strictly contained in $D$ for all $t \in [0,T]$, i.e.,
\begin{equation}\notag
    \operatorname{supp} c(\cdot,t) \subset D, \qquad \forall t \in [0,T].
\end{equation}
We introduce the time-dependent partial differentiation operator $\mathcal{L}_2$ as follows:
\begin{equation}\notag
    \mathcal{L}_2 = 
\nabla \cdot (\mathscr{A}_2 
\nabla ) + c,
\end{equation}
where the complex-valued coefficients $\mathscr{A}_2(\mathbf{x},t) = \big(a_2^{ij}(\mathbf{x},t)\big)_{i,j=1}^d$ are given by the piecewise definition:
\begin{equation}\notag
    \mathscr{A}_2(\mathbf{x},t) =
    \begin{cases}
        \mathcal{B}(t), & (\mathbf{x},t) \in \overline{D} 	\times [0,T], \\[4pt]
        \mathbf{I}, & (\mathbf{x},t) \in (\mathbb{R}^d \setminus \overline{D}) 	\times [0,T].
    \end{cases}
\end{equation}
Here  $\mathcal{B}(t) = \big(b^{ij}(t)\big)_{i,j=1}^d$ is a Hermitian matrix such that $\mathcal{B} \in L^{\infty}[0,T]$. Moreover, the matrix $\mathcal{B}(t)$ satisfies the uniform ellipticity condition:
\begin{equation}\notag
    b^{ij}(t)\,\xi_i\,\overline{\xi_j} \ge 	\theta |\xi|^2, \quad \forall\, \xi \in \mathbb{C}^d, \; \forall\, t \in [0,T].
\end{equation}
with $\theta \in \mathbb R_+$.

Let the unknown function $U: \mathbb{R}^d 	\times [0,T] \rightarrow \mathbb{C}$ satisfy the following Cauchy problem for a nonlinear Schr\"{o}dinger equation:
\begin{equation}\label{1.2}
    \begin{cases}
        \mathrm{i}\partial_t U + \mathcal{L}_2 U + N(U) = 0 & 	\text{in } \mathbb{R}^{d} 	\times (0,T], \\
        U(\mathbf{x},0) = \Phi(\mathbf{x}) & 	\text{in } \mathbb{R}^d,
    \end{cases}
\end{equation}
where the nonlinearity $N: \mathbb{C} 	\to \mathbb{C}$ is a complex polynomial defined as:
\begin{equation}\label{eq:nonlinearity}
    N(U) = \sum_{k=2}^{l_{0}} \alpha_k U^k,
\end{equation}
with $l_0 \geq 2$ being a fixed integer. The complex-valued coefficients $\alpha_k$ satisfy the smoothness and localization conditions:
\begin{equation}\notag
    \alpha_k \in C^{\infty}([0,\infty) 	\times \mathbb{R}^d; \mathbb{C}), \quad k=2,\dots,l_0,
\end{equation}
with their spatial support strictly contained in $D$:
\begin{equation}\notag
    \operatorname{supp}(\alpha_k(\cdot,t)) \subset D, \quad \forall t \in [0,\infty).
\end{equation}
The smooth initial data $\Phi \in C^{\infty}(\mathbb{R}^d)$ will be specified in Section~\ref{section 3}.

	%In the context of wave physics, the coefficients of the above operators $ L_i $ ($ i = 1, 2 $) characterize the medium through which the wave propagates. The tensor $ \mathscr{A}_i $ describes the medium's inherent capacity for wave transmission and its directional dependence (anisotropy). The vector field $ \mathbf{b}_i $ represents the advective velocity of the medium, which transports the wave. The scalar coefficient $ c_i $ accounts for energy dissipation and shifts in the wave's natural frequency. In acoustics, for a homogeneous, isotropic fluid, the tensor $ \mathscr{A}_i $ simplifies to $ c^2 \mathbf{I} $, where $ c $ denotes the speed of sound in the medium.
We are now in a position to state our main theorems. For the initial-boundary value problem of the linear Schr\"odinger equation \eqref{1.1}, we have:

\begin{theorem}\label{thm:main}
For any arbitrarily large $\mathcal{M} > 0$, $n$ distinct points $\mathbf{x}_1,\mathbf{x}_2,\ldots,\mathbf{x}_n \in \partial D$, and any time interval $[0,T]$ with $0 < T < \infty$, by prescribing appropriate initial and boundary inputs (as specified in \eqref{initial and boundary condition}), there exists a small parameter $r_0$ given by \eqref{4.18} such that problem \eqref{1.1} admits a unique solution $u$ satisfying
\[
u \in L^{\infty}\big([0,T];H^{3}_{\mathrm{loc}}(\Omega)\big),
\]
with the following properties:
\begin{enumerate}
    \item \textit{Interior improved H\"older regularity}: for any bounded Lipschitz domain $\Omega' \Subset \Omega$,
    \[
    u \in L^{\infty}\big([0,T]; C^{1,\frac{1}{2}}(\overline{\Omega}')\big).
    \]

    \item \textit{Simultaneous localized gradient amplification near the points}
    $\mathbf{x}_1,\mathbf{x}_2,\ldots,\mathbf{x}_n$ \textit{in the following sense}: for almost every $t \in [0,T]$ and for each $i=1,2,\ldots,n$,
    \begin{equation*}
        \|\nabla u(\cdot,t)\|_{C^{0}\!\left(B_{3r_0}(\mathbf{x}_i)\setminus\overline{D}\right)}
        \ge \mathcal{M},
    \end{equation*}
    and
    \begin{equation*}
        \frac{\|u(\cdot,t)\|_{C^{1,\frac12}\!\left(B_{3r_0}(\mathbf{x}_i)\setminus\overline{D}\right)}}
        {\|u(\cdot,t)\|_{C^{1,\frac12}(\overline{D})}}
        > \frac{\mathcal{M}}{2}.
    \end{equation*}

    \item \textit{Vanishing measure of the localized gradient amplification set}: for almost every $t \in [0,T]$,
    \begin{equation*}
        \operatorname{meas}\left\{
        \mathbf{x} \in \bigcup_{i=1}^{n}\left(B_{3r_0}(\mathbf{x}_i)\setminus \overline{D}\right)
        : |\nabla u(\mathbf{x},t)| > \mathcal{M}
        \right\}
        = O(\mathcal{M}^{-d})
        \qquad \text{as } \mathcal{M}\to\infty,d=2, 3.
    \end{equation*}
\end{enumerate}
\end{theorem}

Similarly, for the Cauchy problem of the nonlinear Schr\"odinger-type equation \eqref{1.2}, we have:

\begin{theorem}\label{thm:main2}
For any arbitrarily large $\mathcal{M} > 0$, $n$ distinct points $\mathbf{x}_1,\mathbf{x}_2,\ldots,\mathbf{x}_n \in \partial D$, and any time interval $[0,T]$ with $0 < T < \infty$, by prescribing appropriate initial inputs (as specified in \eqref{initial condition}), there exists a small parameter $r_0$ given by \eqref{eeq 5.20} such that problem \eqref{1.2} admits a unique solution $U$ satisfying
\[
U \in L^{\infty}\big(0,T;H^{3}_{\mathrm{loc}}(\mathbb{R}^d)\big),
\]
with the following properties:
\begin{enumerate}
    \item \textit{H\"older regularity}: for any bounded Lipschitz domain $\Omega' \subset \mathbb{R}^d$,
    \[
    U \in L^{\infty}\big([0,T]; C^{1,\frac{1}{2}}(\overline{\Omega}')\big).
    \]

    \item \textit{Simultaneous localized gradient amplification near the points}
    $\mathbf{x}_1,\mathbf{x}_2,\ldots,\mathbf{x}_n$ \textit{in the following sense}: for almost every $t \in [0,T]$ and for each $i=1,2,\ldots,n$,
    \begin{equation*}
        \|\nabla U(\cdot,t)\|_{C^{0}\!\left(B_{3r_0}(\mathbf{x}_i)\setminus\overline{D}\right)}
        \ge \mathcal{M},
    \end{equation*}
    and
    \begin{equation*}
        \frac{\|U(\cdot,t)\|_{C^{1,\frac12}\!\left(B_{3r_0}(\mathbf{x}_i)\setminus\overline{D}\right)}}
        {\|U(\cdot,t)\|_{C^{1,\frac12}(\overline{D})}}
        > \frac{\mathcal{M}}{2}.
    \end{equation*}

    \item \textit{Vanishing measure of the localized gradient amplification set}: for almost every $t \in [0,T]$,
    \begin{equation*}
        \operatorname{meas}\left\{
        \mathbf{x} \in \bigcup_{i=1}^{n}\left(B_{3r_0}(\mathbf{x}_i)\setminus \overline{D}\right)
        : |\nabla U(\mathbf{x},t)| > \mathcal{M}
        \right\}
        = O(\mathcal{M}^{-d})
        \qquad \text{as } \mathcal{M}\to\infty,d=2, 3.
    \end{equation*}
\end{enumerate}
\end{theorem}

    \begin{remark}
    In Theorems \ref{thm:main} and \ref{thm:main2}, we identify and construct a mechanism that produces highly localized gradient amplification of solutions to both the linear and nonlinear Schr\"{o}dinger-type equations in two and three spatial dimensions. The phenomenon is described by the following two features:

\begin{enumerate}
    \renewcommand{\labelenumi}{(\roman{enumi})}
    \item
    For any prescribed set of $n$ distinct points 
    $\mathbf{x}_1,\mathbf{x}_2,\ldots,\mathbf{x}_n \in \partial D$, 
    where $D$ is the support of the anisotropic coefficients 
    $\mathscr{A}_1$ in \eqref{1.1} or $\mathscr{A}_2$ in \eqref{1.2}, 
    the zero-order term $c$ in \eqref{1.2}, and the nonlinear term $N$ in \eqref{1.2},
    we rigorously prove---by selecting suitably tuned smooth initial--boundary data
    for \eqref{1.1} or smooth initial data for \eqref{1.2}---that there exist
    localized gradient amplification regions in neighborhoods of each $\mathbf{x}_i$
    located outside $D$ ($i=1,\ldots,n$).
    In these regions, the $C^{0}$-norms of the spatial gradients of the solutions,
    namely $\nabla u(t)$ to \eqref{1.1} and $\nabla U(t)$ to \eqref{1.2},
    exceed any prescribed large constant $\mathcal{M}$ for almost every
    $t\in[0,T]$. Moreover, the ratio between the local $C^{1,\alpha}$-norm
    of the solution in these regions and the $C^{1,\alpha}$-norm inside $D$
    is bounded from below by $\mathcal{M}/2$.
    These properties are established in Theorems~\ref{thm:main}
    and~\ref{thm:main2}, where the time horizon $T\in\mathbb{R}_+$ is arbitrary.

    \item
    We further provide precise quantitative estimates for the measure of the
    localized gradient amplification regions in terms of the amplitude threshold
    $\mathcal{M}$, which characterizes the strength of the gradient amplification
    for the solutions $u$ and $U$ to \eqref{1.1} and \eqref{1.2}, respectively.
    Our analysis shows that, as $\mathcal{M} \to \infty$, the measure of these regions
    tends to zero, demonstrating that the gradient amplification phenomenon is
    sharply localized in space.
\end{enumerate}
    \end{remark}

	  To the best of our knowledge, the present work constitutes the first systematic and customized construction of solutions that exhibit localized gradient amplification for Schr\"{o}dinger-type equations across both linear and nonlinear regimes.

	  %A comprehensive discussion regarding the underlying methodology and its physical implications is provided in Subsection~\ref{subsection 1.3}. This subsection concludes with a series of remarks detailing the spectral techniques employed in our proofs and the rigorous correspondence between the analytical assumptions and the resulting conclusions of Theorems~\ref{thm:main} and~\ref{thm:main2}.

   \begin{remark}
    The mechanism developed in this work for producing customized solutions with localized gradient amplification exhibits several fundamental features:
    \begin{itemize}
        \item[(i)]  The construction is rooted in spectral analysis, specifically leveraging the transmission eigenfunctions associated with the interior transmission problem \cite{ColtonKress} and their approximation via Herglotz wave functions. By exploiting these spectral tools, we construct a precisely tuned auxiliary function $u_0$ (see Theorem~\ref{thm:main_result}), which serves as the essential building block for prescribing the initial and boundary inputs in both the linear and nonlinear regimes. Detailed construction procedures are provided in Section~\ref{section 3}.

        \item[(ii)]  A crucial aspect of our results is that the gradient amplification occurs within the framework of classical derivatives. Both the solution $u(t)$ to the linear problem \eqref{1.1} and $U(t)$ to the nonlinear problem \eqref{1.2} exhibit local $H^3$-regularity for almost every $t \in [0,T]$. By the Sobolev embedding theorem, this ensures that the solutions possess $C^{1,\frac{1}{2}}$-regularity in the neighborhoods of the amplification points. Consequently, the spatial gradients remain well-defined in the classical sense even as their magnitudes exceed the arbitrarily large threshold $\mathcal{M}$.

        \item[(iii)]  The localized gradient amplification is not a transient phenomenon but persists over any prescribed time interval $[0, T]$ for $T \in \mathbb{R}_+$. While this follows from classical well-posedness in the linear case, the nonlinear setting requires a more refined analysis. By carefully tuning the initial data and applying a Banach fixed-point argument, we ensure the stability and existence of the solution $U(t)$ over an arbitrarily large time horizon $T$, as further detailed in Remark~\ref{remark 5.3} and the proof of Theorem~\ref{thm:main2}.
    \end{itemize}
\end{remark}

	The remainder of this paper is organized as follows. Section 2 establishes the requisite mathematical preliminaries. In Section 3, we develop a unified framework for constructing initial and boundary inputs for both problems \eqref{1.1} and \eqref{1.2}. Section 4 presents the complete proof of Theorem~\ref{thm:main_result}, establishing well-posedness for the linear auxiliary problem \eqref{remaining term} and constructing solutions with localized gradient amplification through strategic parameter tuning. Section 5 extends this methodology to the nonlinear setting in Theorem~\ref{thm:main2}, where a fixed-point argument establishes well-posedness for the auxiliary problem \eqref{nonlinear hyperbolic equation with source term and zero initial condition} and proves the analogous localized gradient amplification behavior of the solution.

	\section{Preliminaries}
	
	Throughout this paper, the constant $C$ may change from line to line. In the statements of theorems, we explicitly specify the parameters on which constants depend. However, within proofs, for brevity and to avoid cumbersome notation, we typically do not detail the precise dependence of constants on various parameters. For particularly significant constants that play a crucial role in the argument, we assign specific labels such as $C_1$, $C_2$, and so forth. When the dependence on certain key parameters is essential to the analysis, we may highlight this relationship through notation such as $C_2(r_0)$ to indicate that the constant $C_2$ depends specifically on the parameter $r_0$.

	Let $J_m$ denote the Bessel function of the first kind and $j_m$ the spherical Bessel function. Let $j_{m,s}$ denote the $s$-th positive zero of $J_m$ and $j_{m,s}'$ the $s$-th positive zero of $J_m'$. According to \cite{Abramowitz} and \cite{LiuZou}, the following properties hold:
	\begin{equation*}
		m \leq j_{m,1}' < j_{m,1} < j_{m,2}' < j_{m,2} < j_{m,3}' < \cdots.
	\end{equation*}
	The Bessel function admits the Weierstrass factorization:
	\begin{equation*}\label{eq:bessel_factorization}
		J_m(x) = \frac{(|x|/2)^m}{\Gamma(m+1)} \prod_{s=1}^{\infty} \left(1 - \frac{x^2}{j_{m,s}^2}\right).
	\end{equation*}
	For integer or positive $\alpha$, the Bessel function has the power series expansion:
	\begin{equation}\label{eq:bessel_series}
		J_\alpha(x) = \sum_{m=0}^{\infty} \frac{(-1)^m}{m!\,\Gamma(m+\alpha+1)} \left(\frac{x}{2}\right)^{2m+\alpha}, \quad x > 0,
	\end{equation}
	where $\Gamma(x)$ is the gamma function. The relationship between $J_m$ and $j_m$ is given by:
	\begin{equation}\label{eq:bessel_spherical}
		j_m(x) = \sqrt{\frac{\pi}{2x}} J_{m+\frac{1}{2}}(x).
	\end{equation}
	The spherical harmonics $Y_m^l(\theta,\varphi)$ (see \cite[Theorem 2.8]{ColtonKress}) are defined as:
	\begin{equation}\label{eq:spherical_harmonics}
		Y_m^l(\theta, \varphi) = \sqrt{\frac{2m+1}{4\pi} \frac{(m-|l|)!}{(m+|l|)!}} P_m^{|l|}(\cos \theta) e^{i m \varphi},
	\end{equation}
	where $m = 0,1,2,\ldots$, $l = -m,\ldots,m$, and $P_m^{|l|}$ are the associated Legendre functions. The following estimate for Legendre functions will be used in the subsequent analysis. 
	\begin{lemma}\label{lem:legendre_bound}
		(see \cite{Lohofer}, Corollary 3). For real $x \in [-1,1]$ and integers $m$, $n$ with $1 \leq m \leq n$,
		\begin{equation*}
			\frac{1}{\sqrt{2.22(m+1)}} < \max_{x \in [-1,1]} |P_n^m(x)| \sqrt{\frac{(n-m)!}{(n+m)!}} < \frac{2^{5/4}}{\pi^{3/4}} \frac{1}{m^{1/4}}.
		\end{equation*}
	\end{lemma}
		In what follows, we assume that $\omega$ is a fixed constant. The Herglotz function is defined as
	\begin{equation}\label{Herglotz}
		H_{g}(\mathbf{x})=\int_{\mathbb S^{d-1}} g(\theta) \exp(\mathrm{i} \omega \mathbf{x} \cdot \theta) d\theta, \quad \mathbf{x} \in \mathbb{R}^d,
	\end{equation}
	where $g \in L^2(\mathbb{S}^{d-1})$ is the Herglotz kernel. We now state a key lemma that will be used in the next section.
	
	\begin{lemma}(see~\cite{Weck})\label{Herglotz lemma}
		Let $\mathcal{O}$ be a bounded domain of class $C^{\alpha,1}$, $\alpha \in \mathbb{N} \cup \{0\}$, in $\mathbb{R}^d$. Denote by $\mathbb{T}$ the space of all Herglotz functions of the form (\ref{Herglotz}). Define
		\[
		\mathbb{T}(\mathcal{O}) := \left\{\left.V\right|_{\mathcal{O}} : V \in \mathbb{H}\right\}
		\]
		and
		\begin{equation*}\label{wave number}
			\mathfrak{H}(\mathcal{O}) := \left\{u \in C^{\infty}(\mathcal{O}) : \Delta V + \kappa^2 V = 0 \text{ in } \mathcal{O}\right\},
		\end{equation*}
		where $\kappa\in \mathbb R_+$. Then $\mathbb{T}(\mathcal{O})$ is dense in $\mathbb{T}(\mathcal{O}) \cap H^{\alpha+1}(\mathcal{O})$ with respect to the $H^{\alpha+1}(\mathcal{O})$-norm.
	\end{lemma}
	
	\begin{remark}\label{remark 2.3}
		The Herglotz function $H_g$ defined in \eqref{Herglotz} satisfies the Helmholtz equation $\Delta H_g + \omega^2 H_g = 0$ in $\mathbb{R}^d$. By standard elliptic regularity theory, $H_g$ is analytic with respect to the spatial variable. A key property of Herglotz functions is that they are generally not square-integrable over the entire space; specifically, $H_g \in L^{2}_{\text{loc}}(\mathbb{R}^d)$ but $H_g \notin L^{2}(\mathbb{R}^d)$ for non-trivial kernels $g$. 
	\end{remark}

	\section{A general framework for specifying inputs}\label{section 3}
	
	In this section, we establish the general framework for constructing initial and boundary inputs for problems \eqref{1.1} and \eqref{1.2}. To achieve this, we employ spectral analysis, interior transmission problems, and Herglotz approximation techniques to construct a function $u_0$ that will serve as the foundation for our desired inputs.
	
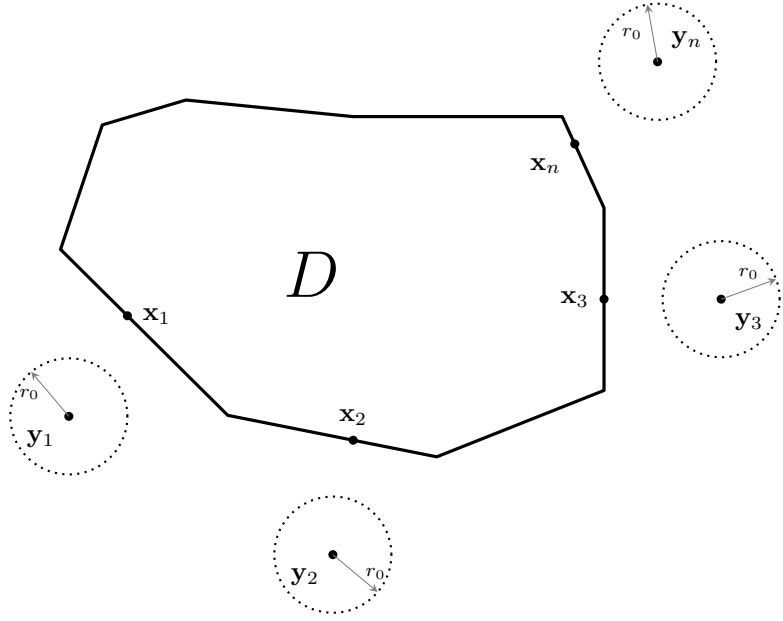
\begin{figure}[htbp]
	\centering
	% Define professional styles
	\tikzset{
		main poly/.style={draw=black, line width=1.2pt, line join=round}, 
		dotted circle/.style={draw=black, dotted, line width=0.8pt},
		point dot/.style={circle, fill=black, inner sep=1.2pt},
		math font/.style={font=\small},
		radius line/.style={draw=gray, thin, ->, >=stealth}
	}
	
	\begin{tikzpicture}[scale=1.1]
		
		% --- 1. Define Constants ---
		\def\rzero{0.7} % Radius of the neighborhood
		\def\dist{1.4}  % FIXED DISTANCE from xi to yi
		
		% --- 2. Define Region D Vertices ---
		\coordinate (v1) at (-3, 0.5);
		\coordinate (v2) at (-2.5, 2);
		\coordinate (v3) at (-1.5, 2.3);
		\coordinate (v4) at (0.5, 2.1);
		\coordinate (v5) at (3, 2.1);
		\coordinate (v6) at (3.5, 1);
		\coordinate (v7) at (3.5, -1.2);
		\coordinate (v8) at (1.5, -2);
		\coordinate (v9) at (-1, -1.5);
		
		% Draw Boundary of D
		\draw[main poly] (v1) -- (v2) -- (v3) -- (v4) -- (v5) -- (v6) -- (v7) -- (v8) -- (v9) -- cycle;
		\node[font=\Huge] at (0, 0.2) {$D$};
		
		% --- 3. Points xi and yi (Mathematically Aligned) ---
		
		% x1 and y1 (Left side)
		\path (v9) -- (v1) coordinate[pos=0.6] (x1);
		\node[point dot, label={[math font]right:$\mathbf{x}_1$}] at (x1) {};
		\coordinate (y1) at ([shift={(240:\dist)}]x1);
		\draw[dotted circle] (y1) circle (\rzero);
		\node[point dot, label={[math font]below left:$\mathbf{y}_1$}] at (y1) {};
		\draw[radius line] (y1) -- ++(130:\rzero) node[midway, left, font=\tiny] {$r_0$};
		
		% x2 and y2 (Bottom side)
		\path (v8) -- (v9) coordinate[pos=0.4] (x2);
		\node[point dot, label={[math font]above:$\mathbf{x}_2$}] at (x2) {};
		\coordinate (y2) at ([shift={(-100:\dist)}]x2);
		\draw[dotted circle] (y2) circle (\rzero);
		\node[point dot, label={[math font]below left:$\mathbf{y}_2$}] at (y2) {};
		\draw[radius line] (y2) -- ++(-40:\rzero) node[midway, right, font=\tiny] {$r_0$};
		
		% x3 and y3 (Right side)
		\path (v6) -- (v7) coordinate[pos=0.5] (x3);
		\node[point dot, label={[math font]left:$\mathbf{x}_3$}] at (x3) {};
		\coordinate (y3) at ([shift={(0:\dist)}]x3);
		\draw[dotted circle] (y3) circle (\rzero);
		\node[point dot, label={[math font]below right:$\mathbf{y}_3$}] at (y3) {};
		\draw[radius line] (y3) -- ++(20:\rzero) node[midway, above, font=\tiny] {$r_0$};
		
		% xn and yn (Top Right side)
		\path (v5) -- (v6) coordinate[pos=0.3] (xn);
		\node[point dot, label={[math font]below left:$\mathbf{x}_n$}] at (xn) {};
		\coordinate (yn) at ([shift={(45:\dist)}]xn);
		\draw[dotted circle] (yn) circle (\rzero);
		\node[point dot, label={[math font]above right:$\mathbf{y}_n$}] at (yn) {};
		\draw[radius line] (yn) -- ++(100:\rzero) node[midway, left, font=\tiny] {$r_0$};
		
	\end{tikzpicture}
	\caption{Geometric configuration where each external point $\mathbf{y}_i$ is placed at a uniform distance from its corresponding boundary point $\mathbf{x}_i \in \partial D$.}
	\label{fig:geometry_uniform}
\end{figure}
	
	\begin{theorem}\label{thm:main_result}
		For any predetermined positive parameters $\varepsilon < 1$, $r_0$ sufficiently small, and $n$ distinct points $\mathbf{x}_1,\dots,\mathbf{x}_n \in \partial D$, we can properly choose corresponding $n$ balls satisfying:
		\begin{enumerate}
			\renewcommand{\labelenumi}{(\roman{enumi})}
			\item For problem \eqref{1.1} $($bounded domain $\Omega$ with $D \Subset \Omega):$ $B_{r_0}(\mathbf{y}_i) \Subset \Omega \setminus \overline{D}$,
			\item For problem \eqref{1.2} $($full space $\mathbb{R}^d ):$ $B_{r_0}(\mathbf{y}_i) \Subset \mathbb{R}^d \setminus \overline{D}$,
		\end{enumerate}
		such that $|\mathbf{x}_i - \mathbf{y}_i| = 2r_0$, the balls are mutually disjoint and
		$$
		\operatorname{dist}(\partial D, B_{r_0}(\mathbf{y}_i))>\frac{r_0}{4},\ \forall i=1,2...,n.
		$$
	Then, there exists an integer $M := M(r_0)$ and a smooth function $u_0$ with the following properties:
		\begin{enumerate}
			\item $u_0$ has the form$:$
			\begin{equation*}
				u_0(\mathbf{x},t) = H_g(\mathbf{x})\exp(-\mathrm{i} t),
			\end{equation*}
			where $ H_g(\mathbf{x})$ is a Herglotz function given by \eqref{eq Hg}. 
			
			\item $u_0$ satisfies the free Schr\"odinger equation$:$
			\[
			\mathrm{i}\partial_{t} u_0 + \Delta u_0 = 0 \quad \text{in } \mathbb{R}^d \times [0,\infty).
			\]
			
			\item It exhibits small $C^{1,\frac{1}{2}}$-norm in $\overline{D}:$
			\begin{equation}\label{u0 smalles}
				||u_0(\cdot,t)||_{C^{1,\frac{1}{2}}(\overline{D})} < C_1\varepsilon \quad \text{for all}~ t \in [0,\infty),
			\end{equation}
			where constant $C_1$ depends on $D$ and $d$.
			
			\item For any predetermined integer $m \geqslant M(r_0,\omega)$, there exist points $\mathbf{y}^*_i \in \partial B_{r_0}(\mathbf{y}_i)$  ($i=1,\dots,n$) with large amplitude $u_0(\mathbf{y}^*_i,t)$ such that
			\begin{align}
				|u_0(\mathbf{y}^*_i,t)| &> \frac{\sqrt{2m+2}\,r_0^{-1}}{2\sqrt{2\pi}} - C_2\varepsilon, & d &= 2, \label{2 d} \\
				|u_0(\mathbf{y}^*_i,t)| &> \frac{\sqrt{2m+3}\,r_0^{-3/2}}{16} - C_2\varepsilon, & d &= 3, \label{3 d}
			\end{align}
			for all $t \in [0,\infty)$, where constant $C_2$ depends on $r_0$, $D$, and $d$.
		\end{enumerate}
	\end{theorem}
	
	\begin{remark}\label{cf 3.5}
		From now and remaining of this paper, we denote by $\varepsilon$, $r_0$ the parameters, and by $u_0$ the function, specified in Theorem \ref{thm:main_result}.
	\end{remark}
	
	\begin{remark}\label{remark 3.2}
		For notation convenience, we denote $C_2 := C_2(r_0)$. In particular, by \eqref{important estimate for H_g}, \eqref{3.17}, and the Sobolev embedding theorem, $C_2 \to \infty$ as $r_0 \to 0$. The constants $C_1$ and $C_2$ will be used to determine $\varepsilon$, $r_0$, and $m$ in the proofs of Theorems \ref{thm:main} and \ref{thm:main2}.
	\end{remark}
	
	\begin{remark}
		The construction distinguishes between two physical settings:
		\begin{enumerate}
			\renewcommand{\labelenumi}{(\roman{enumi})}
			\item For the linear problem \eqref{1.1} in the bounded domain $\Omega$ (with $D \Subset \Omega$), the balls satisfy $B_{r_0}(\mathbf{y}_i) \Subset \Omega \setminus \overline{D}$.
			\item For the nonlinear problem \eqref{1.2} in the full space $\mathbb{R}^d$, the balls satisfy $B_{r_0}(\mathbf{y}_i) \Subset \mathbb{R}^d \setminus \overline{D}$.
		\end{enumerate}
		Importantly, the selection of these balls is highly flexible they can be chosen arbitrarily within the respective domains, subject only to the above containment conditions. Despite these domain differences, the constructed $u_0$ shares identical properties in both cases. For simplicity, we do not distinguish between the corresponding $u_0$ functions in subsequent analysis.
	\end{remark}
	
	Prior to proving Theorem \ref{thm:main_result}, we establish the following auxiliary result:
	\begin{lemma}\label{I_1,I_2,I_3,I_4}
		For any fixed $r_0$, the series $I_{i}(m,r_0)\rightarrow 0$ as $m \rightarrow +\infty$ for $i=1,2,3,4$, where
		\begin{align*}
			I_1(m,r_0) &:= \sum_{k=1}^{\infty}\frac{(-1)^k\Gamma(m+\frac{3}{2})\sqrt{2m+3} }{k!\Gamma(m+k+\frac{3}{2})}\left(\frac{ r_0}{2}\right)^{2k}r_0^{-\frac{3}{2}}, \\
			I_2(m,r_0) &:= \sum_{k=1}^{\infty}\sum_{k_1+k_2=k}\left(\frac{\Gamma^2(m+\frac{3}{2})}{k_1!k_2!\Gamma(m+k_1+\frac{3}{2})\Gamma(m+k_2+\frac{3}{2})}\right)  \\
			& \qquad \times(-1)^k\left(\frac{ r_0}{2}\right)^{2k}\frac{2m+3}{2m+2k+3}, \\
			I_3(m,r_0) &:= \sum_{k=1}^{\infty}\frac{(-1)^k}{k!}\frac{\Gamma(m+1)\sqrt{2m+2}}{\Gamma(m+k+1)}\left(\frac{ r_0}{2}\right)^{2k}r^{-1}_0, \\
			I_4(m,r_0) &:= \sum_{k=1}^{\infty}\left(\sum_{k_1+k_2=k}\frac{\Gamma^{2}(m+1)}{k_1!k_2!\Gamma(m+k_1+1)\Gamma(m+k_2+1)}\right) \nonumber \\
			& \qquad \times(-1)^k\left(\frac{ r_0}{2}\right)^{2k}\frac{2m+2}{2m+2k+2}.
		\end{align*}
		\begin{proof}
			Set $\nu=m+\frac{1}{2}$, so that $\Gamma(m+k+\frac{3}{2})=\Gamma(\nu+k+1)$ and $\sqrt{2m+3}=\sqrt{2\nu+2}$. Using \eqref{eq:bessel_series}, we obtain
			\begin{equation}\label{I_1 1}
				I_1=\sqrt{2\nu +2} r_0^{-\frac{3}{2}}[\Gamma(\nu+1)(\frac{r_0}{2})^{-\nu}J_{\nu}(r_0)-1].
			\end{equation}
			Let $\beta=(\frac{r_0}{2})^2$ and $\gamma_k(\nu)=\frac{1}{(\nu+k)(\nu+k-1)\cdots(\nu+1)}$. By the Gamma function recurrence and \eqref{eq:bessel_series},
			\begin{equation}\label{I_1 2}
				\begin{aligned}
					|\Gamma(\nu+1)(\frac{r_0}{2})^{-\nu}J_{\nu}(r_0)-1|&=|\sum_{k=1}^{\infty}\frac{(-1)^k}{k!}\beta^k\gamma_k(\nu)|\\
					&\leqslant\exp(\frac{\beta}{\nu})-1\leqslant \frac{\beta}{\nu}\exp(\frac{\beta}{\nu})\leqslant \frac{\beta}{\nu}\exp(\beta),
				\end{aligned}
			\end{equation}
			where we used $|\gamma_k(\nu)|\leqslant \nu^{-k}$. Combining \eqref{I_1 1} and \eqref{I_1 2} yields
			\begin{equation*}
				|I_1|\leqslant \frac{\sqrt{2\nu +2}}{\nu} r_0^{-\frac{3}{2}}\beta\exp(\beta) =\frac{\sqrt{2m +3}}{m+\frac{1}{2}} r_0^{-\frac{3}{2}}(\frac{ r_0}{2})^2\exp(\frac{ r_0}{2})^2\rightarrow 0 
			\end{equation*}
			as $m \rightarrow \infty$.
			
			For $I_2$, rewrite it as
			\[
			\begin{aligned}
				I_2\left(m, r_0\right)=(2 m+3) \sum_{k=1}^{\infty} \frac{(-1)^k\left(r_0 / 2\right)^{2 k}}{2 m+2 k+3} \cdot S_k(m),
			\end{aligned}
			\]
			where 
			\[
			S_k(m)=\sum_{\substack{k_1+k_2=k \\ k_1, k_2 \geq 0}} \frac{1}{k_{1}!k_{2}!} \cdot \frac{\Gamma^2\left(m+\frac{3}{2}\right)}{\Gamma\left(m+k_1+\frac{3}{2}\right) \Gamma\left(m+k_2+\frac{3}{2}\right)}.
			\]
			For fixed $j \geq 0$, the Gamma recurrence implies
			\[
			\frac{\Gamma\left(m+\frac{3}{2}\right)}{\Gamma\left(m+j+\frac{3}{2}\right)} \leq \frac{1}{\left(m+\frac{3}{2}\right)^j}.
			\]
			Applying this to $S_k(m)$ and noting $\sum_{k_1+k_2=k} \frac{1}{k_{1}!k_{2}!}=\frac{2^k}{k!}$, we get
			\[
			S_k(m) \leq \frac{2^k}{k!\left(m+\frac{3}{2}\right)^k}.
			\]
			Thus,
			\begin{equation*}
				\begin{aligned}
					\left|I_2(m)\right| &\leq(2 m+3) \sum_{k=1}^{\infty} \frac{1}{2 m+2 k+3} \cdot \frac{\left( r_0^2 / 2\right)^k}{k!\left(m+\frac{3}{2}\right)^k}\\
					&\leqslant \sum_{k=1}^{\infty} \frac{\left( r_0^2 / 2\right)^k}{k!\left(m+\frac{3}{2}\right)^k}=\exp(\frac{ r_0^2}{2m+3})-1\rightarrow 0
				\end{aligned}
			\end{equation*}
			as $m \rightarrow \infty$. The same approach applies to $I_3$ and $I_4$, so their proofs are omitted.
			
			The proof is complete.
		\end{proof}
	\end{lemma}
	
	\begin{proof}[Proof of Theorem \ref{thm:main_result}]
		The proof proceeds in two steps.
		
		\medskip
		\textbf{Step 1: Construction of the transmission eigenfunction and its properties.}
		
		Consider the interior transmission problem (see \cite{ColtonKress}, p.~319):
		\begin{equation}\label{Interior Transmission Problem}
			\left\{\begin{array}{@{}ll@{}}
				(\Delta +  \mathfrak{n}^2) w = 0 & \text{in } \displaystyle\bigcup_{i=1}^{n} B_{r_0}(\mathbf{y}_i), \\
				(\Delta + 1) v = 0 & \text{in } \displaystyle\bigcup_{i=1}^{n} B_{r_0}(\mathbf{y}_i), \\
				w = v,\ \dfrac{\partial w}{\partial \nu} = \dfrac{\partial v}{\partial \nu} & \text{on } \partial\left(\displaystyle\bigcup_{i=1}^{n} B_{r_0}(\mathbf{y}_i)\right),
			\end{array}\right.
		\end{equation}
		where $\nu$ denotes the outward unit normal. A pair $(w,v)$ satisfying \eqref{Interior Transmission Problem} is called a transmission eigenfunction corresponding to the transmission eigenvalue $\omega$. We analyze explicit solutions in both two-dimensional  and three-dimensional case, showing that $v(\mathbf{x})$ attains its maximum simultaneously on each $\partial B_{r_0}(\mathbf{y}_i)$ ($i=1,\dots,n$) with arbitrarily large magnitude for large predetermined parameter $m$.
		
		\medskip
		\item[\textbf{Case 1}: $d=3.$] For a fixed positive integer $m$ to be determined, in polar coordinates with $\mathbf{x} \in B_{r_0}(\mathbf{y}_i)$ ($i=1,\dots,n$):
		\begin{equation*}
			w(\mathbf{x}) = \alpha_{m}^{l} j_m( \mathfrak{n} r_i) Y_{m}^{l}(\theta_i, \varphi_i), \quad
			v(\mathbf{x}) = \beta_{m}^{l} j_m( r_i) Y_{m}^{l}(\theta_i, \varphi_i), \quad -m \leqslant l \leqslant m,
		\end{equation*}
		where $r_i = |\mathbf{x}-\mathbf{y}_i|$, $\theta_i = \arccos((x_3-y_{i3})/r_i)$, $\varphi_i = \arctan((x_2-y_{i2})/(x_1-y_{i1}))$. To satisfy the boundary conditions on $\partial B_{r_0}(\mathbf{y}_i)$, we set $\alpha^l_m = [j_m( r_0)/j_m( \mathfrak{n} r_0)] \beta^l_m$ with $\mathfrak{n}$ solving
		\begin{equation*}
			\mathfrak{n} j_{m}^\prime( \mathfrak{n} r_0) j_{m}( r_0) - j_{m}^\prime(r_0) j_{m}(\mathfrak{n} r_0) = 0.
		\end{equation*}
		Normalizing $v$ such that $\|v\|_{L^2(B_{r_0}(\mathbf{y}_i))} = 1$ for all $i$, we obtain
		\begin{equation*}
			\beta_m^l = \sqrt{\frac{2}{\pi}} \frac{1}{\sqrt{\int_0^{ r_0} J^2_{m+\frac{1}{2}}(r) r  dr}}.
		\end{equation*}
		Combining these results yields
		\begin{equation}\label{v in d=3}
			v(\mathbf{x}) = \frac{r_i^{-\frac{1}{2}} J_{m+\frac{1}{2}}( r_i)}{\sqrt{\int_0^{ r_0} J_{m+\frac{1}{2}}^2(r) r  dr}} Y_m^l(\theta_i, \varphi_i), \quad \mathbf{x} \in B_{r_0}(\mathbf{y}_i).
		\end{equation}
		By using recurrence relation for Bessel function, we have
		\begin{equation*}
			\left(\frac{J_{m+\frac{1}{2}}( r)}{r^{\frac{1}{2}}}\right)^{\prime}=\frac{m J_{m+\frac{1}{2}}( r)- r J_{m+\frac{3}{2}}( r) }{r^{\frac{3}{2}}}.
		\end{equation*}
		Let $m> r_0$ and use the fact that $J_{m+\frac{1}{2}}(r)>J_{m+\frac{3}{2}}(r)~ ,\forall r \in (0,m]$, then $J_{m+\frac{1}{2}}( r) r^{-\frac{1}{2}}$ is increasing on $\left(0, r_0\right]$. Consequently, by \eqref{eq:bessel_series}, we calculate
		\begin{equation}\label{max d=3}
			\begin{aligned}
				&\max _{r_i \in\left(0, r_0\right]}\frac{ {r_i^{-\frac{1}{2}}} J_{m+\frac{1}{2}}( r_i)}{\sqrt{\int_0^{ r_0} J_{m+\frac{1}{2}}^2(r) r d r}} =\frac{{r_0^{-\frac{1}{2}}} J_{m+\frac{1}{2}}( r_0)}{\sqrt{\int_0^{ r_0} J_{m+\frac{1}{2}}^2(r) r d r}}\\
				&=\frac{\sum_{k=0}^{\infty}\frac{(-1)^{k}r_{0}^{m+2k}}{k!2^{m+2k+\frac{1}{2}}\Gamma(m+k+\frac{3}{2})}}{\sqrt{\sum_{k=0}^{\infty}\sum_{k_1+k_2=k}\frac{1}{k_1{!}k_2{!}\Gamma(m+k_1+\frac{3}{2})\Gamma(m+k_2+\frac{3}{2})}\frac{(-1)^k(r_0)^{2m+2k+3}}{2^{2m+2k+1}2m+2k+3}}}\\
				&=\frac{\sqrt{2m+3}r_0^{-\frac{3}{2}}+\sum_{k=1}^{\infty}\frac{(-1)^k\Gamma(m+\frac{3}{2})\sqrt{2m+3} }{k!\Gamma(m+k+\frac{3}{2})}\left(\frac{r_0}{2}\right)^{2k}r_0^{-\frac{3}{2}}}{\sqrt{1+\sum_{k=1}^{\infty}\sum_{k_1+k_2=k}\left(\frac{\Gamma^2(m+\frac{3}{2})}{k_1!k_2!\Gamma(m+k_1+\frac{3}{2})\Gamma(m+k_2+\frac{3}{2})}\right)\times\frac{(-1)^k( r_0/2)^{2k}(2m+3)}{(2m+2k+3)}}}\\
				&=\frac{\sqrt{2m+3}r_0^{-\frac{3}{2}}+I_1(m,r_0)}{\sqrt{1+I_2(m, r_0)}}.
			\end{aligned}
		\end{equation}
		By Lemma \ref{I_1,I_2,I_3,I_4}, $I_1, I_2 \to 0$ as $m \to \infty$ for fixed $r_0$. Thus, there exists $M_1(r_0)$ such that for predetermined parameter $m \geqslant M_1$,
		\begin{equation}\label{upper bound in d=3}
	 \sqrt{2m+3}  r_0^{-\frac{3}{2}}	\geq	\frac{J_{m+\frac{1}{2}}( r_0)  r_0^{-\frac{1}{2}}}{\sqrt{\int_0^{ r_0} J_{m+\frac{1}{2}}^2(r) r  dr}} > \frac{1}{2} \sqrt{2m+3}  r_0^{-\frac{3}{2}}.
		\end{equation}
		Combining \eqref{eq:spherical_harmonics}, \eqref{v in d=3}, \eqref{upper bound in d=3}, and spherical harmonic properties from Lemma \ref{lem:legendre_bound}, there exist points $\mathbf{y}^*_i \in \partial B_{r_0}(\mathbf{y}_i)$ satisfying
		\begin{equation}\label{3.12}
			|v(\mathbf{y}^*_i)| > \frac{1}{16} \sqrt{2m+3}  r_0^{-\frac{3}{2}}.
		\end{equation}
		
		\medskip
		\item[\textbf{Case 2}: $d=2.$] The two-dimensional case follows analogously. For $\mathbf{x} \in B_{r_0}(\mathbf{y}_i)$ ($i=1,\dots,n$):
		\begin{equation*}
			w(\mathbf{x}) = \alpha_m J_m( \mathfrak{n} r_i) e^{\mathrm{i}m \theta_i}, \quad
			v(\mathbf{x}) = \beta_m J_m( r_i) e^{\mathrm{i}m \theta_i},
		\end{equation*}
		with $r_i = |\mathbf{x}-\mathbf{y}_i|$, $\theta_i = \arctan((x_2-y_{i2})/(x_1-y_{i1}))$. Setting $\alpha_m = [J_m( r_0)/J_m( \mathfrak{n} r_0)] \beta_m$ and $	\beta _m=\frac{1}{\sqrt{2\pi}}\frac{1}{\sqrt{ \int_{0}^{ r_0}J^2_m(r)rdr}}$  yields 
		\begin{equation}\label{d=2}
			v(\mathbf{x}) = \frac{1}{\sqrt{2\pi}} \frac{ J_m( r_i)}{\sqrt{\int_{0}^{ r_0} J_{m}^2(r) r  dr}} e^{\mathrm{i}m \theta_i}.
		\end{equation}
		For $m > r_0$, by \eqref{eq:bessel_series}, we obtain
		\begin{equation}\label{max d=2}
			\max_{r_i \in [0,r_0]} \frac{ J_m(r_i)}{\sqrt{\int_{0}^{ r_0} J_{m}^2(r) r  dr}} = \frac{\sqrt{2m+2} r^{-1}_0 + I_3(m,r_0)}{\sqrt{1 + I_4(m,r_0)}}.
		\end{equation}
		By Lemma \ref{I_1,I_2,I_3,I_4}, $I_3, I_4 \to 0$ as $m \to \infty$ for fixed $r_0$. Then, for any predetermined $m > M_2(r_0)$, there exist points $\mathbf{y}^{*}_{i} \in \partial B{r_0}(\mathbf{y}_i)$ such that
		\begin{equation}\label{3.14}
			|v(\mathbf{y}^*_i)| \geqslant \frac{\sqrt{2m+2} r^{-1}_0}{2\sqrt{2\pi}}.
		\end{equation}
		Defining $M = \max\{M_1(r_0), M_2(r_0), \lfloor r_0 \rfloor + 1\}$, both \eqref{3.12} and \eqref{3.14} hold for predetermined parameter $m \geqslant M$.

		\medskip
		\textbf{Step 2: Herglotz approximation and construction of $u_0$.}
		
	Since $D$ is a bounded Lipschitz domain, it admits arbitrarily close smooth outer approximations; see \cite{Antonini}. In particular, there exists a $C^\infty$-smooth bounded domain $D_{r_0/4}$ such that
		\begin{equation*}
			D \Subset D_{r_0/4}
			\quad \text{and} \quad
			\sup_{x \in \partial D_{r_0/4}} \operatorname{dist}(x,D) < \frac{r_0}{4}.
		\end{equation*}
		As established in Theorem \ref{thm:main_result}, the sets \(D_{r_0/4}\) and \(B_{r_0}(\mathbf{y}_i)\) are disjoint for all \(i\). Then we define $\mathcal{O} = D_{r_0/4} \cup \bigcup_{i=1}^{n} B_{r_0}(\mathbf{y}_i)$. Consider the interior transmission problem
		\begin{equation}\label{Interior transmission problem 2}
			\left\{\begin{array}{@{}ll@{}}
				(\Delta +  \mathfrak{n}^2) w_1 = 0 & \text{in } \mathcal{O}, \\
				(\Delta + 1) v_1 = 0 & \text{in } \mathcal{O}, \\
				w_1 = v_1,\ \dfrac{\partial w_1}{\partial \nu} = \dfrac{\partial v_1}{\partial \nu} & \text{on } \partial \mathcal{O}.
			\end{array}\right.
		\end{equation}
		The pair $(w_1, v_1) = \chi_{\bigcup_{i=1}^n B_{r_0}(\mathbf{y}_i)} (w,v) + \chi_{D_{r_0/4}} (0,0)$ satisfies \eqref{Interior transmission problem 2} by construction. Setting $\omega=1$ and applying Lemma \ref{Herglotz lemma} to $v_1=\chi_{\bigcup_{i=1}^n B_{r_0}(\mathbf{y}_i)} v + \chi_{D_{r_0/4}} \cdot 0$, there exists $g \in L^2(\mathbb{S}^{d-1})$ such that the Herglotz function
		\begin{equation}\label{eq Hg}
			H_g(\mathbf{x}) = \int_{\mathbb{S}^{d-1}} g(\theta) \exp(\mathrm{i} \mathbf{x} \cdot \theta)  d\theta
		\end{equation}
		satisfies
		\begin{equation}\label{eq3.14}
			\left\| H_g - \left( \chi_{\bigcup_{i=1}^n B_{r_0}(\mathbf{y}_i)} v + \chi_{D_{r_0/2}} \cdot 0 \right) \right\|_{H^5(\mathcal{O})} < \varepsilon.
		\end{equation}
		This implies
		\begin{equation}\label{important estimate for H_g}
			\| H_g - v \|_{H^5(B_{r_0}(\mathbf{y}_i))} < \varepsilon \quad (i=1,\dots,n) \quad \text{and} \quad \| H_g \|_{H^5(D)} < \varepsilon.
		\end{equation}
		Define
		\begin{equation}\label{u_0}
			u_0(\mathbf{x},t) = H_g(\mathbf{x}) \exp(-\mathrm{i}  t).
		\end{equation}
	Since $H_g(\mathrm{x})$ satisfies $(\Delta+1)H_g(\mathrm{x})=0$ in $\mathbb{R}^d$, a straightforward calculation shows that the function $u_0$ defined above satisfies 
		\[
		\mathrm{i}\partial_t u_0 + \Delta u_0 = 0 \quad \text{in } \mathbb{R}^d \times [0,+\infty),
		\] 
		which implies that Properties 1 and 2 hold. Furthermore, by the Sobolev embedding theorem and the estimate \eqref{important estimate for H_g}, there exist constants $C_1$ (depending on $D$ and $d$) and $C_2$ (depending on $r_0$, $D$, and $d$) such that
		\begin{equation}\label{3.18}
			\| H_g \|_{C^{1,\frac{1}{2}}(\overline{D})} \leqslant C_1 \varepsilon,
		\end{equation}
		and
		\begin{equation}\label{3.17}
			\| H_g - v \|_{L^{\infty}(B_{r_0}(\mathbf{y}_i))} \leqslant C_2	\| H_g - v \|_{H^5(B_{r_0}(\mathbf{y}_i))}\leqslant C_2 \varepsilon \quad (i=1,\dots,n).
		\end{equation}	
		Combining \eqref{3.12}, \eqref{3.14}, \eqref{u_0}, \eqref{3.18}, and \eqref{3.17}, together with the continuity of $H_g$ and $v$, we conclude that Properties 3 and 4 are verified. 
	
		The proof is complete.
	\end{proof}
	
	\begin{remark}
		The approximation of $\chi_{\bigcup_{i=1}^n B_{r_0}(\mathbf{y}_i)} v + \chi_{D_{r_0/2}} \cdot 0$ in the $H^5$ norm required in \eqref{eq3.14} is valid due to the $C^\infty$ regularity of $\partial D_{r_0/2}$. This construction serves a critical purpose: the resulting estimate will be employed in the subsequent section to separately control the source term $F_1$ in equation~\eqref{remaining term},  the source term $F_2$ in equation~\eqref{nonlinear hyperbolic equation with source term and zero initial condition}, and all spatial derivatives $\partial_{\mathbf{x}}^{\alpha} u_0$ for $|\alpha|\leqslant 3$. Detailed justification for this approach is provided in \eqref{estimate on F_1} and Lemma~\ref{Lemma 5.1}.
	\end{remark}
	
	The initial and boundary data for problem \eqref{1.1} are formally specified as:
	\begin{equation}\label{initial and boundary condition}
		\begin{aligned}
			& \phi(\mathbf{x}) = \left. u_0(\mathbf{x},t) \right|_{\Omega \times \{ t=0 \}} = H_{g}(\mathbf{x})|_{\Omega}, \\
			& \psi(\mathbf{x},t) = u_0 \left. (\mathbf{x},t) \right|_{\partial \Omega \times [0,T]} = \exp(-\mathrm{i} t)H_g(\mathbf{x})|_{\partial \Omega \times [0,T]}.
		\end{aligned}
	\end{equation}
	Similarly, the initial data for problem \eqref{1.2} are defined as:
	\begin{equation}\label{initial condition}
	  \Phi(\mathbf{x}) = \left. u_0(\mathbf{x},t) \right|_{\mathbb{R}^d \times \{ t=0 \}} = H_{g}(\mathbf{x})|_{\mathbb{R}^d}, 
	\end{equation}
	
	\begin{remark}
		By Remark \eqref{remark 2.3} and the assumption that $\partial \Omega$ is Lipschitz, we have $\phi \in C^{\infty}(\Omega)$, $\psi \in C^\infty([0, T]; C(\partial \Omega))$, and $\Phi \in C^{\infty}(\mathbb{R}^d)$.
	\end{remark}

	\begin{remark}\label{remark 3.8}
		It is essential to recognize that the constructions in \eqref{initial and boundary condition} and \eqref{initial condition} remain abstract at this stage. The function $u_0$--and consequently all associated data--depends critically on the yet-to-be-determined parameters $\varepsilon$, $m$, $r_0$, and the configuration of points $\mathbf{x}_1, \dots, \mathbf{x}_n$ on $\partial D$. A concrete realization of $u_0$ requires explicitly specifying these parameters, which are intrinsically determined by the prescribed large constant $\mathcal{M}$ and the spatial locations of the points $\mathbf{x}_1, \dots, \mathbf{x}_n$. In the next two sections, we will strategically   choose  $\varepsilon$, $m$, and $r_0$ in terms of $\mathcal{M}$ and the given boundary point configuration, ultimately yielding a well-defined $u_0$ and corresponding initial and boundary inputs suitable for rigorous analysis.
	\end{remark}
	
	\section{Proof of Theorem \ref{1.1}}\label{sec:4}
	
	In this section, we present the complete proof of Theorem~\ref{thm:main}. Our strategy involves analyzing an auxiliary problem \eqref{remaining term} with zero initial-boundary conditions and a source term, derived from the linear Schr\"odinger equation \eqref{1.1}. We establish well-posedness for this  auxiliary problem via the standard Galerkin method and demonstrate that its solution exhibits small H\" older norm on $\Omega \setminus \overline{D}$. Finally, by integrating Theorems \ref{thm:main_result} and \ref{3.1} and a delicate balancing of the parameters $\varepsilon$, $r_0$, and $m$, we precisely generate the localized gradient amplification behavior central to our main result.
	
	Let $\mathcal{U} : \Omega \times [0,T] \rightarrow \mathbb{C}$ satisfy 
	\begin{equation}\label{remaining term}
		\begin{cases}
			\mathrm{i}\partial_t \mathcal{U}+\mathcal{L}_1\mathcal{U} = F_1 & \text{in } \Omega \times (0,T], \\
			\mathcal{U}(\mathbf{x},0) = 0 & \text{in } \Omega, \\
			\mathcal{U}(\mathbf{x},t) = 0 & \text{on } \partial\Omega \times [0,T],
		\end{cases}
	\end{equation}
	where the source term
	\begin{equation}\label{F_1}
		F_1(\mathbf{x}, t)=\nabla \cdot\left((\mathbf{I}-\mathscr{A}_1) \nabla u_0\right).
	\end{equation}
Recalling that $\mathscr{A}_1 - \mathbf{I}$ has compact support in $D$, we obtain
\begin{equation*}
	\operatorname{supp} F_1(\cdot,t) \subset D, \quad \forall t \in [0,T].
\end{equation*}
	Furthermore, combining \eqref{important estimate for H_g}, \eqref{u_0}, \eqref{F_1}, and the assumption $\mathscr{A}_1 \in W^{2,\infty}(\Omega)$, we obtain the following estimate for $F_1$:
	\begin{equation}\label{estimate on F_1,}
		 \| F_1 \|_{L^\infty(0,T;H^1(\Omega))} \leqslant C \varepsilon,
	\end{equation}
	and
	\begin{equation}\label{estimate on F_1}
		\| F_1 \|_{H^2(0,T;L^2(\Omega))} + \| F_1 \|_{H^1(0,T;H^1(\Omega))} 
		\leqslant C \|H_g\|_{H^3(D)} \leqslant C\varepsilon.
	\end{equation}
	Here, the spatial derivatives of \( u_0 \) appearing in \( F_1 \) are of order at most two, whereas the space--time Sobolev norm estimate requires third-order spatial derivatives of \( H_g \) due to the presence of spatial gradients in the \( H^1(\Omega) \) norm. Hence, the \( H^3 \) norm of \( H_g \) suffices to control all terms in the estimate. The constant $C$ depends on $d$, $T$, $D$, and $\|\mathscr{A}_1\|_{W^{2,\infty}(\Omega)}$.

	Next, we obtain the well-posedness for auxiliary problem $\eqref{remaining term}$ by energy estimates and Galerkin method. 
	\begin{theorem}\label{3.1}
		For any $T$ with $0<T<\infty$, problem \eqref{remaining term} admits a unique solution $\mathcal{U}$ satisfying
		$$ 	\mathcal{U}\in L^{\infty}\left(0,T;H^{3}_{\mathrm{loc}}(\Omega)\right)$$ 
		with the following properties:
		\begin{enumerate}
			\item Interior improved H\"older regularity on $\Omega:$ for any  Lipschtiz domain $\Omega^\prime \Subset \Omega$,
			\begin{equation*}
				\mathcal{U} \in L^{\infty}(0,T;C^{1,\frac{1}{2}}(\overline{\Omega}^{\prime})). 
			\end{equation*} 
			\item  Smallness estimate:
			\begin{equation*}
				\|\mathcal{U}\|_{L^{\infty}(0,T;C^{1,\frac{1}{2}}(\overline{\Omega}^{\prime})) }\leqslant C_3 \varepsilon ,
			\end{equation*}
			where the constant $C_3$ depends on where the constant $C_3$ depends on $d$, $T$, $\theta$, $D$, $\Omega^\prime$, $\Omega$, $\|\mathscr{A}_1\|_{W^{2,\infty}(\Omega)}$,$\|\mathscr{A}_1\|_{C^{1,1}(\overline{\Omega})}$.
		\end{enumerate}
	\end{theorem}

	\begin{proof}
		\textbf{Step 1: a priori estimates for problem \eqref{remaining term}.}
		
		Introduce the notation
		\[
		A[\mathcal{U},\mathcal{V};t]
		:=\int_\Omega \sum_{i,j=1}^d a_1^{ij}(\cdot)\, 
		\mathcal{U}_{x_i}\, \overline{\mathcal{V}_{x_j}}.
		\]
		Multiplying \eqref{remaining term} by $\overline{\mathcal{U}}$, using integration by parts, we obtain
		\begin{equation}\label{eq4.4}
			\mathrm{i}\int_{\Omega}\partial_{t}\mathcal{U}\, \overline{\mathcal{U}}\,dx
			-\int_{\Omega}a^{ij}_{1}(x)\partial_{i}\mathcal{U}\,
			\overline{\partial_{j}\mathcal{U}}\,dx
			=\int_{\Omega}F_1\,\overline{\mathcal{U}}\,dx .
		\end{equation}
		
		Taking the complex conjugate of \eqref{eq4.4}, subtracting it from \eqref{eq4.4}, taking the imaginary part, and applying the Cauchy--Schwarz inequality yields:
		\[
		\frac{d}{dt}\|\mathcal{U}(t)\|_{L^{2}(\Omega)}^{2}
		\le C\Bigl(\|\mathcal{U}(t)\|_{L^{2}}^{2}
		+\|F_{1}(t)\|_{L^{2}}^{2}\Bigr).
		\]
		By Grönwall’s inequality,
		\begin{equation}\label{eq 4.55}
			\|\mathcal{U}(t)\|_{L^{2}(\Omega)}^{2}
			\le C\int_{0}^{t}\|F_1(s)\|_{L^{2}(\Omega)}^{2}\,ds ,
		\end{equation}
		for all $0\le t\le T$.
	
		Multiplying \eqref{remaining term} by $\overline{\partial_t\mathcal{U}}$, and integrating by parts, we obtain
		\begin{equation}\label{eq4.5}
			\mathrm{i}\int_{\Omega}|\partial_{t}\mathcal{U}|^{2}\,dx
			-\int_{\Omega}a^{ij}_{1}(x)\partial_{i}\mathcal{U}\,
			\partial_{t}\overline{\partial_{j}\mathcal{U}}\,dx
			=\int_{\Omega}F_1\,\overline{\partial_t\mathcal{U}}\,dx.
		\end{equation}
		Note that
		\[
		\int_{\Omega}\overline{
			a^{ij}_1(x)\partial_i\mathcal{U}\,
			\partial_{t}\overline{\partial_{j}\mathcal{U}}}\,dx
		=
		\int_{\Omega}a^{ji}_{1}(x)\,\overline{\partial_{i}\mathcal{U}}\,
		\partial_{t}\partial_{j}\mathcal{U}\,dx
		=
		\int_{\Omega}a^{ij}_{1}(x)\,
		\partial_{t}\partial_{i}\mathcal{U}\,
		\overline{\partial_{j}\mathcal{U}}\,dx .
		\]
		Taking the complex conjugate of \eqref{eq4.5}, adding it back, and taking the real part gives
		\[
		\frac{d}{dt}A[\mathcal{U},\mathcal{U};t]
		= -2\Re\int_{\Omega}F_1(t)\,
		\overline{\partial_t\mathcal{U}(t)}\,dx .
		\]
		Integrating over $[0,t]$, using integration by parts in time, the identity $\mathcal{U}(0)=0$, and \eqref{eq 4.55}, we calculate
		\[
		\begin{aligned}
			A[\mathcal{U},\mathcal{U};t]
			&= A[\mathcal{U},\mathcal{U};0]
			-2\Re\int_{0}^{t}\int_{\Omega}
			F_1(x,s)\,\overline{\partial_s\mathcal{U}(x,s)}\,dx\,ds 
			\\
			&= -2\Re\int_{\Omega}F_1(x,t)\overline{\mathcal{U}(x,t)}\,dx
			+2\Re\int_{0}^{t}\int_{\Omega}
			\partial_sF_1(x,s)\,\overline{\mathcal{U}(x,s)}\,dx\,ds
			\\
			&\le C\left(\|F_1(t)\|_{L^{2}}^{2}
			+\|\mathcal{U}(t)\|_{L^{2}}^{2}
			+\int_{0}^{t}\|F_1^{\prime}(s)\|_{L^{2}}^{2}
			+\|\mathcal{U}(s)\|_{L^{2}}^{2}\,ds \right)
			\\
			&\le C\|F_1\|_{W^{1,2}(0,T;L^{2}(\Omega))}^{2}.
		\end{aligned}
		\]
		Thus,
		\begin{equation}\label{eq 4.8}
			\|\mathcal{U}(t)\|_{H^{1}_{0}(\Omega)}
			\le C\,\|F_1\|_{W^{1,2}(0,T;L^{2}(\Omega))},
		\end{equation}
		for all $t\in[0,T]$.
		
		Next, set $\mathcal{U}_{1}:=\partial_t\mathcal{U}$. Then $\mathcal{U}_{1}$ satisfies
		\begin{equation}\label{eq4.8}
			\begin{cases}
				\mathrm{i}\partial_t\mathcal{U}_{1}+\mathcal{L}_1\mathcal{U}_{1}=F_1', &\text{in }\Omega\times(0,T],\\[2mm]
				\mathcal{U}_{1}(x,0)=-\mathrm{i}F_{1}(x,0), &\text{in }\Omega,\\[2mm]
				\mathcal{U}_{1}(x,t)=0, &\text{on }\partial\Omega\times[0,T].
			\end{cases}
		\end{equation}
		Since $F_1$ has compact support, $\mathcal{U}_1(x,0)|_{\partial\Omega}
		= -\mathrm{i}F_1(x,0)|_{\partial\Omega}=0$, which is compatible with the boundary condition. Repeating the same procedure as above, multiplying \eqref{eq4.8} by $\overline{\mathcal{U}_1}$, integrating by parts, we obtain
		\begin{equation}\label{eq4.9}
			\mathrm{i}\int_{\Omega}\partial_{t}\mathcal{U}_1\,
			\overline{\mathcal{U}_1}\,dx
			-\int_{\Omega}a^{ij}_{1}(x)\partial_{i}\mathcal{U}_1\,
			\overline{\partial_{j}\mathcal{U}_1}\,dx
			=\int_{\Omega}F_1'\,\overline{\mathcal{U}_1}\,dx .
		\end{equation}
		Taking the conjugate of \eqref{eq4.9}, subtracting, and taking imaginary parts again yields
		\[
		\frac{d}{dt}\|\mathcal{U}_1(t)\|_{L^{2}}^{2}
		\le C\Bigl(\|\mathcal{U}_1(t)\|_{L^{2}}^{2}
		+\|F_1'(t)\|_{L^{2}}^{2}\Bigr),
		\]
		and by Grönwall's inequality,
		\begin{equation}\label{eq 4.10}
			\|\mathcal{U}_1(t)\|_{L^{2}}^{2}
			\le C\int_{0}^{t}
			\|F_1'(s)\|_{L^{2}}^{2}\,ds ,
		\end{equation}
		for all $0\le t\le T$. Multiplying \eqref{eq4.8} by $\overline{\partial_t\mathcal{U}_1}$, integrating by parts, we obtain
		\begin{equation*}\label{eq4.11}
			\mathrm{i}\int_{\Omega}|\partial_{t}\mathcal{U}_1|^{2}\,dx
			-\int_{\Omega}a^{ij}_{1}(x)\partial_{i}\mathcal{U}_1\,
			\partial_{t}\overline{\partial_{j}\mathcal{U}_1}\,dx
			=\int_{\Omega}F_1'\,\overline{\partial_t\mathcal{U}_1}\,dx .
		\end{equation*}
		Proceeding as before, we derive
		\[
		\frac{d}{dt}A[\mathcal{U}_1,\mathcal{U}_1;t]
		= -2\Re\int_{\Omega}F_1'(t)\,
		\overline{\partial_t \mathcal{U}_1(t)}\,dx .
		\]
		Integrating over $[0,t]$, integrating by parts in $t$, and using
		$\mathcal{U}_1(0)=-\mathrm{i}F_1(x,0)$ together with \eqref{eq 4.10}, we obtain
		\begin{equation}\label{eq 4.12}
			\begin{aligned}
				A[\mathcal{U}_1,\mathcal{U}_1;t]
				&=A[\mathcal{U}_1,\mathcal{U}_1;0]
				-2\Re\int_{0}^{t}\int_{\Omega}
				F_1'(x,s)\,\overline{\partial_s\mathcal{U}_1(x,s)}\,dx\,ds 
				\\
				&=A[\mathcal{U}_1,\mathcal{U}_1;0]
				-2\Re\int_{\Omega}
				F_1'(x,t)\overline{\mathcal{U}_1(x,t)}
				-F_1'(x,0)\overline{\mathcal{U}_1(x,0)}\,dx
				\\
				&\quad
				+2\Re\int_{0}^{t}\int_{\Omega}
				F_1''(x,s)\,\overline{\mathcal{U}_1(x,s)}\,dx\,ds
				\\
				&\le C\Bigl(\|\nabla F_1(0)\|_{L^{2}}^{2}
				+\|F_1'(t)\|_{L^{2}}^{2}
				+\|F_1'(0)\|_{L^{2}}^{2}
				+\|\mathcal{U}_1(t)\|_{L^{2}}^{2}
				+\|\mathcal{U}_1(0)\|_{L^{2}}^{2}
				\\
				&\quad
				+\int_{0}^{t}\|F_1''(s)\|_{L^{2}}^{2}
				+\|\mathcal{U}_1(s)\|_{L^{2}}^{2}\,ds\Bigr)
				\\
				&\le C\Bigl(
				\|F_1\|_{H^{1}(0,T;H^{1}(\Omega))}^{2}
				+\|F_1\|_{H^{2}(0,T;L^{2}(\Omega))}^{2}
				\Bigr).
			\end{aligned}
		\end{equation}
		Combining \eqref{eq 4.10} and \eqref{eq 4.12}, and using the coercivity
		$A[\mathcal{U}_1,\mathcal{U}_1;t]\ge \theta\|\nabla \mathcal{U}_1(t)\|_{L^{2}}^{2}$, we deduce
		\begin{equation}\label{eq 4.14}
			\|\mathcal{U}_1(t)\|_{H^{1}(\Omega)}^{2}
			\le C\Bigl(\|F_1\|_{H^{1}(0,T;H^{1}(\Omega))}^{2}
			+\|F_1\|_{H^{2}(0,T;L^{2}(\Omega))}^{2}\Bigr),
		\end{equation}
		for all $0\le t\le T$.
	
		\textbf{Step 2: existence, uniqueness and Hölder regularity for problem \eqref{remaining term}.}
		
		Using the Galerkin method and the a priori estimate \eqref{eq 4.8}, we obtain a weak solution
		\[
		\mathcal{U}\in L^{\infty}(0,T;H^{1}_{0}(\Omega))
		\]
		to \eqref{remaining term}.  
		Uniqueness follows from Grönwall’s inequality. By \eqref{estimate on F_1} and \eqref{eq 4.14}, we also have
		\[
		\partial_t \mathcal{U} \in L^{\infty}(0,T;H_0^{1}(\Omega))
		\]
		and
		\[
		\|\partial_t \mathcal{U}\|_{L^{\infty}(0,T;H_0^{1}(\Omega))}
		\le C\left(
		\|F_1\|_{H^{1}(0,T;H^{1}(\Omega))}
		+\|F_1\|_{H^{2}(0,T;L^{2}(\Omega))}
		\right)
		\le C\varepsilon .
		\]
		Then, by the above estimate and \eqref{estimate on F_1,}, the solution
		$\mathcal{U}(t)\in H^{1}(\Omega)$ satisfies
		\[
		\mathcal{L}_1\mathcal{U}(t)
		=
		F_1(t)-\mathrm{i}\,\partial_t \mathcal{U}(t)
		\in H^{1}(\Omega)
		\]
		in the weak sense for almost every $t\in[0,T]$.
		Since $\mathscr{A}_1\in C^{1,1}(\overline{\Omega})$, interior elliptic
		regularity (Theorem~8.10 in~\cite{GilbargTrudinger}) implies that, for any
		$\Omega'\Subset\Omega\setminus\overline{D}$,
		\[
		\mathcal{U}(t)\in H^{3}(\Omega'),
		\qquad
		\|\mathcal{U}(t)\|_{H^{3}(\Omega')}
		\le C\Bigl(
		\|\partial_t\mathcal{U}(t)\|_{H^{1}(\Omega)}
		+\|F_1(t)\|_{H^{1}(\Omega)}
		\Bigr)
		\le C\varepsilon ,
		\]
		for almost every $t\in[0,T]$.
		Taking the supremum over $t\in[0,T]$ yields
		\[
		\|\mathcal{U}\|_{L^{\infty}(0,T;H^{3}(\Omega'))}
		\le C\varepsilon .
		\]
		Finally, by the Sobolev embedding theorem, properties~(1) and~(2) follow.

		The proof is complete.
	\end{proof}

	We now proceed to establish Theorem \ref{thm:main}.
	\begin{proof}[Proof of Theorem \ref{thm:main}.]
		In the three-dimensional case, according to Remark \ref{remark 3.8}, we select predetermined parameters $r_0$, $\varepsilon$, and $m$ as follows:
		\begin{equation}\label{4.18}
			\begin{aligned}
				&\hfill r_0 = \min\left\{\,
				\frac{1}{3}(\mathcal{M} + 1)^{-1},\ 
				\frac{1}{6}\min_{1 \le i < j \le n} |\mathbf{x}_i - \mathbf{x}_j|,\ 
				\frac{1}{3}\operatorname{dist}(D, \partial \Omega)
				\,\right\}, \\
				&\hfill \varepsilon = \min\left\{\frac{1}{C_1 + C_2(r_0) + C_3},1\right\}, \\
				&\hfill m = \max\left\{ \left\lfloor 512 r_0^3 \right\rfloor + 1,  M(r_0) \right\},
			\end{aligned}
		\end{equation}
		where the constants $C_1$, $C_2(r_0)$, and $M(r_0)$ are specified in Theorem \ref{thm:main_result}, and $C_3$ is given by Theorem \ref{3.1}. Then, we choose the initial data $\phi$ and boundary data $\psi$ for problem \eqref{1.1} as specified in \eqref{initial and boundary condition}, where the corresponding $u_0$ is determined by the parameters $r_0$, $\varepsilon$, and $m$ in \eqref{4.18}. Let $\mathcal{U}$ denote the unique solution to auxiliary problem \eqref{remaining term}. A direct calculation verifies that 
		\begin{equation}\label{eq 4.17}
			u := u_0 + \mathcal{U}
		\end{equation}	
		constitutes the unique solution to problem \eqref{1.1}. By properties 2 of Theorem \ref{remaining term} and the smoothness of $u_0$, it follows that $u \in L^\infty(0,T;C^{1,\frac{1}{2}}(\overline{\Omega^\prime}))$ for any $\Omega^\prime \Subset \Omega$ with Lipschitz boundary. Applying the mean value theorem to \eqref{u0 smalles} and \eqref{3 d} in Theorem \ref{thm:main_result} yields points $\mathbf{z}_i \in B_{3r_0}(\mathbf{x}_i) \backslash \overline{D}$ ($i=1,2,\dots,n$) such that for all $t \in [0,T]$:
		\begin{equation}\label{4.22}
			|\nabla u_0(\mathbf{z}_i,t)| \geqslant \frac{
				\tfrac{1}{16}\sqrt{2m+3} \, r_0^{-3/2} - \left(C_1 + C_2(r_0)\right)\varepsilon
			}{3r_0}.
		\end{equation}
		Combining \eqref{4.18}, \eqref{eq 4.17}, \eqref{4.22}, and properties (2) of Theorem \ref{3.1}, we obtain
		\begin{equation}\label{large gradient of u}
			\begin{aligned}
				|\nabla u(\mathbf{z}_i,t)| &\geqslant \frac{
					\tfrac{1}{16}\sqrt{2m+3} \, r_0^{-3/2} - \left(C_1 + C_2(r_0)\right)\varepsilon
				}{3r_0} - C_3\varepsilon\\
				& >\frac{
					\tfrac{1}{16}\sqrt{1024r_0^3} \, r_0^{-3/2} - 1
				}{3r_0} - 1=\frac{1}{3r_0}-1>\mathcal{M},
			\end{aligned}
		\end{equation}
		for a.e. $t \in [0,T]$. Here we choose $m \geqslant \left\lfloor 512 r_0^3 \right\rfloor + 1 $ in \eqref{4.18} to guarantee $\tfrac{1}{16}\sqrt{2m+3} \, r_0^{-3/2}\geqslant 2$. On the other hand, using prpoerties (3) in Theorem \ref{thm:main_result} and above estimate, we calculate 
		\begin{equation}\label{quotient of graident u outside D and inside D}
		\frac{|\nabla u(\mathbf{z}_i,t)|}{||u(\cdot,t)||_{C^{1,\frac{1}{2}}(\overline{D})}}\geqslant \frac{\mathcal{M}}{||u_0(\cdot,t)||_{C^{1,\frac{1}{2}}(\overline{D})}+||\mathcal{U}(\cdot,t)||_{C^{1,\frac{1}{2}}(\overline{D})}}\geqslant \frac{\mathcal{M}}{(C_1+C_3)\varepsilon} \geqslant \mathcal{M}>\frac{\mathcal{M}}{2}.
		\end{equation}
		By \eqref{large gradient of u} and \eqref{quotient of graident u outside D and inside D}, properties (1) and (2)  in Theorem \ref{thm:main} are verified. Furthermore, under the above construction, the localized gradient amplification points are confined within the balls $B_{3r_0}(\mathbf{x}_i),i=1,2,...,n,$ which implies that 
		\begin{equation*}
			\operatorname{meas}\big\{ \mathbf{x} \in \bigcup_{i=1}^{n} B_{3r_0}(\mathbf{x}_i)\backslash \overline{D} :|\nabla u(\mathbf{x},t)|  > \mathcal{M} \big\} < 36n\pi r_0^3 < \frac{100n}{(\mathcal{M} + 1)^3}  \to 0
		\end{equation*}
		as $\mathcal{M} \to \infty$, for a.e. $t \in [0,T]$. Properties (3) in Theorem \ref{thm:main} is verified. The two-dimensional case follows analogously.
		
		The proof is complete.
	\end{proof}
	
	\section{Proof of Theorem \ref{thm:main2}}\label{sec:5}

	In this section, we apply identical methodology to the nonlinear Schr\"odinger equation \eqref{1.2}, proving Theorem~\ref{thm:main2}. Building upon the linear framework, we consider a nonlinear auxiliary problem \eqref{nonlinear hyperbolic equation with source term and zero initial condition} with zero initial conditions, derived from the problem \eqref{1.2}. Using a fixed-point argument in a suitably chosen Banach space, we establish existence and uniqueness of equation \eqref{nonlinear hyperbolic equation with source term and zero initial condition} while demonstrating the solution's H\" older norm remains small. Strategic parameter tuning of $\varepsilon$, $r_0$, and $m$ completes the proof.
	
	Let $\mathscr{U}:\mathbb{R}^d\times [0,T]\rightarrow \mathbb{C}$ satisfy
	\begin{equation}\label{nonlinear hyperbolic equation with source term and zero initial condition}
		\left\{
		\begin{array}{@{}l@{\quad}l}
			\mathrm{i}\partial_t \mathscr{U}+\hat{\mathcal{L}}_2 \mathscr{U} + \hat{N}(\mathscr{U}) = F_2 & \text{in } \mathbb{R}^d \times (0,T], \\
			\mathscr{U}(\mathbf{x},0) = 0 & \text{in } \mathbb{R}^{d},
		\end{array}
		\right.
	\end{equation}
	where the following conditions hold:
\begin{enumerate}
	\item
	\[
	\hat{\mathcal{L}}_2\mathscr{U} := \nabla\cdot(\mathscr{A}_2\nabla \mathscr{U}) + \hat{c}\,\mathscr{U}.
	\]
	
	\item
	\[
	\hat{c}(\mathbf{x},t) = c(\mathbf{x},t) + \sum_{k=2}^{l_0} k\,\alpha_k(\mathbf{x},t)\,u_{0}^{\,k-1}(\mathbf{x},t).
	\]
	
	\item The nonlinear term is given by
	\begin{equation}\label{eq:hatN}
		\hat{N}(\mathscr{U})
		= \sum_{k=2}^{l_0} \left(\alpha_k \sum_{i=2}^k \binom{k}{i} u_{0}^{\,k-i} \mathscr{U}^i \right).
	\end{equation}
	
	\item The source term is given by
	\begin{equation}\label{eq:F2}
		F_2(\mathbf{x},t)
		= \nabla\cdot\big(\mathbf{I}-\mathscr{A}_2\big)\nabla u_0
		- c\,u_0 - \sum_{k=2}^{l_0}\alpha_k u_0^k.
	\end{equation}
\end{enumerate}

	Next we give some properties for the coefficients, nonlinear term and source term in \eqref{nonlinear hyperbolic equation with source term and zero initial condition} which will be used in the proof of Theorem \ref{thm:5.2}.
	\begin{lemma}\label{Lemma 5.1}
		The coefficients, nonlinear term and source term in equation \eqref{nonlinear hyperbolic equation with source term and zero initial condition} satisfy the following properties:
		\begin{enumerate}
			\item $\operatorname{supp}\hat{c}(\cdot,t)$, $\operatorname{supp}\hat{N}(\mathscr{U})(\cdot,t)$
			and $\operatorname{supp}F_2(\cdot,t)\subset D$ for all $t\in[0,T]$.
			
			\item For a.e.\ $t\in[0,T]$,
			\[
			\|\hat{c}(\cdot,t)\|_{W^{2,\infty}(D)}\le C,\qquad
			\|F_2(\cdot,t)\|_{H^{3}(D)} \le C\,\varepsilon,
			\]
			where the constant $C$ depends only on $d$, $l_0$, $D$, $\|\mathscr{A}_2\|_{L^\infty(0,T)}$,
			$\|c\|_{L^\infty(0,T;W^{3,\infty}(D))}$, and $\|\alpha_k\|_{L^\infty(0,T;W^{3,\infty}(D))}$.
		\end{enumerate}
	\end{lemma}
	
	\begin{proof}
		Properties 1 follows immediately from the compact support in $D$ of $\mathbf{I} - \mathscr{A}_2(\cdot, t)$, $c(\cdot, t)$, $\alpha_k(\cdot, t)$ for $t \in [0, T]$, $k = 2, \dots, l_0$ combined with the expressions defining $\hat{c}$, $\hat{N}(\mathscr{U})$ and $F_2$.
		
		For properties 2, note that $\alpha_k$ are smooth, so all their derivatives are bounded in $D$. It remains to show $\partial_{\mathbf{x}}^{\alpha} u_0 \in L^{\infty}(D)$ for all multi-indices $\alpha$ with $|\alpha| \leqslant 3$. By the Sobolev embedding theorem, \eqref{important estimate for H_g}, and \eqref{u_0}, we have
		\begin{equation}\label{5.2}
			\|\partial_{\mathbf{x}}^{\alpha} u_0\|_{L^{\infty}(D)} \leqslant C \|\partial_{\mathbf{x}}^{\alpha} u_0\|_{H^{2}(D)} \leqslant C \|H_g\|_{H^5(D)} \leqslant C \varepsilon \leqslant C, \quad \forall |\alpha| \leqslant 3, \quad \forall t \in [0, T],
		\end{equation}
		where the final inequality holds since $\varepsilon \leqslant 1$ (See Theorem \ref{thm:main_result} and Remark \ref{cf 3.5}). Consequently,
		\begin{equation*}
			\|\hat{c}(\cdot, t)\|_{W^{2,\infty}(D)}\leqslant  C \left( \|c_2(\cdot, t)\|_{W^{2,\infty}(D)} + \sum_{k=2}^{l_0} \|u_0\|_{W^{2,\infty}(D)}^{k-1} \right) \leqslant C \left( 1 + \sum_{k=1}^{l_0-1} \varepsilon^{k} \right) \leqslant C,
		\end{equation*}
		for a.e. $t \in [0, T]$.
		
		For the source term, using the Banach algebra property of $H^3(D)$ and
	\eqref{important estimate for H_g}, \eqref{eq:F2}, and since $\varepsilon\le1$,
	we obtain
	\begin{equation}\label{eq:F2_est}
		\|F_2(\cdot,t)\|_{H^3(D)} \le C\sum_{k=1}^{l_0}\|H_g\|_{H^5(D)}^k
		\le C\sum_{k=1}^{l_0}\varepsilon^k \le C\varepsilon,
	\end{equation}
	for a.e.\ $t\in[0,T]$.
	\end{proof}
	
	We now establish the well-posedness of problem \eqref{nonlinear hyperbolic equation with source term and zero initial condition} via Banach's fixed point theorem, combined with an estimate for the nonlinear term $F_2$. This approach yields an explicit characterization of the solution's existence time $T$ in terms of the parameter $\varepsilon$. We show that $T$ admits a uniform positive lower bound independent of $\varepsilon$ and, moreover, becomes arbitrarily large for sufficiently small $\varepsilon$. Furthermore, we prove that the Hölder norm of the solution is bounded by $C \sqrt{\varepsilon}$ for some constant $C > 0$.
	
	\begin{theorem}\label{thm:5.2}
		There exists a time $ T > 0 $ $($given below by \eqref{eq:T_lifespan}$)$, depending on $ \varepsilon $, $ d $, $ l_0 $, $ D $, $ \|\mathscr{A}_2\|_{L^{\infty}(0,T)} $, $ \| c\|_{L^{\infty}(0,T; W^{3,\infty}(D))} $, and $\|\alpha_k\|_{L^{\infty}(0,T; W^{3,\infty}(D))} $, such that problem \eqref{nonlinear hyperbolic equation with source term and zero initial condition} admits a unique solution $ \mathscr{U} $ satisfying:
		$$
		\mathscr{U} \in L^{\infty}(0, T; H^3(\mathbb{R}^d)) 
		$$
		with the following properties:
		\begin{enumerate}
			\item Uniform bounds for $T$: 
			$$
			T\geqslant T_{\mathrm{lower}} 	,
			$$
			where $T_{\mathrm{lower}}$ is given by \eqref{5.17}  and independent of $\varepsilon$.
			\item  H\"older regularity:
			$$
			\mathscr{U} \in L^{\infty}\big([0,T]; C^{1,\frac{1}{2}}(\mathbb{R}^d)\big).
			$$
			\item Smallness estimate:
			$$
			\|\mathscr{U}\|_{L^{\infty}(0,T;C^{1,\frac{1}{2}}(\mathbb{R}^d)) }\leqslant C_6 \sqrt{\varepsilon},
			$$
			where the constant $C_6$ depends only on $d$.
		\end{enumerate}
	\end{theorem}
	\begin{proof}The proof proceeds in four steps.
		
		\medskip
		\textbf{Step 1: Setup of fixed-point framework.} 
		
		Define the Banach space
		$$
		X:=\left\{{u} \in L^{\infty}\left(0, T ; H^1\left(\mathbb{R}^d\right)\right) \right\}
		$$
		with norm
		$$
		\|{u}\|:=\underset{0 \leq t \leq T}{\operatorname{ess} \sup } \|{u}(t)\|_{H^1\left(\mathbb{R}^d\right)}.
		$$
		Introduce the stronger norm 
		$$
		|||{u}|||:=\underset{0 \leq t \leq T}{\operatorname{ess} \sup }\|{u}(t)\|_{H^3\left(\mathbb{R}^d\right)},
		$$
		and the closed subset
		$$
		X_\varepsilon:=\left\{{u} \in X \mid\ ||| {u}||| \leq \sqrt{\varepsilon},{u}(0)=0\right\} .
		$$
		For \(\mathscr{V} \in X_\varepsilon\), define \(\mathscr{U} = \mathscr{F}[\mathscr{V}]\) as the solution to the linear initial-value problem
		\begin{equation}\label{linear}
			\left\{
			\begin{aligned}
				&\mathrm{i}\partial_t \mathscr{U}+\hat{\mathcal{L}}_2 \mathscr{U}= F_2 -  \hat{N}(\mathscr{V}) &&\text{on } \mathbb{R}^d \times (0,T], \\
				&\mathscr{U}(\mathbf{x},0) = 0&& \text{in } \mathbb{R}^d.
			\end{aligned}
			\right.
		\end{equation}
		
		\medskip
		\textbf{Step 2: A priori estimate for the problem \eqref{linear}.}
		
		For problem \eqref{linear}, we claim that we have the estimate 
		\begin{equation}\label{key estimate}
			|||\mathscr{U}|||\leqslant C_4(T)T\varepsilon,
		\end{equation}
		where constant \(C_4(T)\) depends continuously on \(T\), \(d\), \(l_0\), \(D\), \(\|\mathscr{A}_2\|_{L^{\infty}(0,T)}\), \(\| c\|_{L^{\infty}(0,T; W^{3,\infty}(D))}\), and \(\|\alpha_k\|_{L^{\infty}(0,T; W^{3,\infty}(D))}\), but not on \(\varepsilon\). Moreover, \(C_4(T)\) is monotonically increasing in \(T\) with \(C_4(0) < \infty\).
		
			For any multi-index $\beta$ with $0\le|\beta|\le 3$, apply $\partial_{\mathbf{x}}^\beta$
		to \eqref{linear} to obtain
		\begin{equation}\label{eq5.6}
			\left\{
			\begin{aligned}
				&\mathrm{i}\partial_t \partial^{\beta}_{\mathrm{x}}\mathscr{U}+\hat{\mathcal{L}}_2 \partial^{\beta}_{\mathrm{x}}\mathscr{U}+[\partial_\mathbf{x}^\beta,\hat{\mathcal{L}}_2]\mathscr{U}= \partial^{\beta}_{\mathrm{x}}F_2 -  \partial^{\beta}_{\mathrm{x}}\hat{N}(\mathscr{V}) &&\text{on } \mathbb{R}^d \times (0,T], \\
				&\partial^{\beta}_{\mathrm{x}}\mathscr{U}(\mathbf{x},0) = 0&& \text{in } \mathbb{R}^d,
			\end{aligned}
			\right.
		\end{equation}
		where $[\partial_\mathbf{x}^\beta,\hat{\mathcal{L}}_2] =\partial_\mathbf{x}^\beta \hat{\mathcal{L}}_2- \hat{\mathcal{L}}_2\partial_\mathbf{x}^\beta $ denotes the commutator, 
		which is a differential operator of order at most 2 in spatial derivatives.  Multiply the first line in \eqref{eq5.6} with $\overline{\partial^{\beta}_{\mathrm{x}}\mathscr{U}}$, use integration by parts one obtains,
		\begin{equation}\label{eq5.7}
			\begin{aligned}
					\mathrm{i}\int_{\mathbb{R}^d}\partial_{t}\partial^{\beta}_{\mathrm{x}}\mathscr{U} \overline{\partial^{\beta}_{\mathrm{x}}\mathscr{U}}d\mathbf{x}-\int_{\mathbb{R}^d}a^{ij}_{2}(t)\partial_{i}\partial^{\beta}_{\mathrm{x}}\mathscr{U}\overline{\partial_{j}\partial^{\beta}_{\mathrm{x}}\mathscr{U}}d\mathbf{x}&+\int_{\mathbb{R}^d}[\partial_\mathbf{x}^\beta,\hat{\mathcal{L}}_2]\mathscr{U}\overline{\partial^{\beta}_{\mathrm{x}}\mathscr{U}}d\mathbf{x}\\
				&=\int_{\mathbb{R}^d}(\partial^{\beta}_{\mathrm{x}}F_2 -  \partial^{\beta}_{\mathrm{x}}\hat{N}(\mathscr{V}))\overline{\partial^{\beta}_{\mathrm{x}}\mathscr{U}}d\mathbf{x} .
			\end{aligned}
		\end{equation}
		Then taking conjugate in \eqref{eq5.7}, subtracting, taking the imaginary part, using integration by parts and Cauchy Schwartz inequality yields 
		\begin{equation}\label{eq5.10}
			\begin{aligned}
					&\quad \frac{d}{dt}||\partial^{\beta}_{\mathrm{x}}\mathscr{U}(t)||^2_{L^{2}(\mathbb{R}^d)}
				\\
		 &\leqslant C\left(|| \partial^{\beta}_{\mathrm{x}}\mathscr{U}(t)||^{2}_{L^{2}(\mathbb{R}^d)}+|| [\partial_\mathbf{x}^\beta,\hat{\mathcal{L}}_2]\mathscr{U}||^{2}_{L^{2}(\mathbb{R}^d)}+||\partial^{\beta}_{\mathrm{x}} F_2(t)||^{2}_{L^{2}(\mathbb{R}^d)}+||\partial^{\beta}_{\mathrm{x}}\hat{N}(\mathscr{V})||^{2}_{L^{2}(\mathbb{R}^d)}\right)\\
			&\leqslant C \left(|| \mathscr{U}(t)||^{2}_{H^{3}(\mathbb{R}^d)}+||F_2(t)||^{2}_{H^{3}(\mathbb{R}^d)}+||\hat{N}(\mathscr{V})||^{2}_{H^{3}(\mathbb{R}^d)}\right)
			\end{aligned}
		\end{equation}
			Summing over $\beta$ with $0\leqslant|\beta| \leqslant 3$  and applying Gronwall's inequality, we obtain that
		\begin{equation}\label{eq5.11}
			||\mathscr{U}(t)||^2_{H^{3}(\mathbb{R}^d)}\leqslant C \int_{0}^{t}|| F_2(s)||^{2}_{H^{3}(\mathbb{R}^d)}+||\hat{N}(\mathscr{V})||^{2}_{H^{3}(\mathbb{R}^d)}ds
		\end{equation}
		for a.e. $ 0\leqslant t\leqslant T$.
		
		Next, we estimate $||\hat{N}(\mathscr{V})||_{H^{3}(\mathbb{R}^d)}$. By \eqref{eq:hatN},
		$\operatorname{supp}\hat{N}(\mathscr{V}(t))\subset D$. Since $\mathscr{V}\in X_{\varepsilon}$, then $||\mathscr{V}(t)||_{H^3(\mathbb{R}^d)}\leqslant \sqrt{\varepsilon}$ for a.e. $t\in[0,T]$. By \eqref{important estimate for H_g} and \eqref{u_0}, we have $||u_0(t)||_{H^3(D)}\leqslant C\varepsilon$ for all $t \in [0,T]$. Combine the above results and $\varepsilon <1$, we calculate
		\begin{equation}\label{eq5.11 estimate for nonlinear term}
			\begin{aligned}
				||\hat{N}(\mathscr{V}(t))||_{H^{3}(\mathbb{R}^d)}&=	||\hat{N}(\mathscr{V}(t))||_{H^{3}(D)}\\
				&\leqslant C\sum_{k=2}^{l_0}\sum_{i=2}^{k} ||\mathscr{V}(t)||^{i}_{H^{3}(D)}||u_0(t)||^{k-i}_{H^{3}(D)}\\
				&\leqslant C \sum_{k=2}^{l_0}\sum_{i=2}^{k} \varepsilon^{\frac{i}{2}}\varepsilon^{k-i}\leqslant C \varepsilon,
			\end{aligned}
		\end{equation} 
		for a.e. $t \in [0,T]$. Combining estimates \eqref{eq:F2_est}, \eqref{eq5.11}, and \eqref{eq5.11 estimate for nonlinear term}, we obtain the desired bound \eqref{key estimate}. By the Hahn–Banach theorem, together with the preceding estimates (see, for example, Theorem 3.2 in \cite{Sogge}), we conclude that the linear problem \eqref{linear} admits a unique solution $\mathscr{U}\in L^{\infty}(0,T; H^{3}(\mathbb{R}^d)).$

		\medskip
		\textbf{Step 3: Contraction mapping argument and well-posedness of problem \eqref{nonlinear hyperbolic equation with source term and zero initial condition}.}
		
		By properly choosing $T$, we show that 
		\begin{equation}\label{eq 5.13}
			\mathscr{F}: X_\varepsilon \to X_\varepsilon
		\end{equation}
		and 
		\begin{equation}\label{eq5.12}
			\|\mathscr{F}[\mathscr{V}_1]-\mathscr{F}[\mathscr{V}_2]\|\leqslant \frac{1}{2}\|\mathscr{V}_1-\mathscr{V}_2\|, \ \forall \mathscr{V}_1,\mathscr{V}_2 \in X_\varepsilon. 
		\end{equation}
	To ensure that \eqref{eq 5.13} holds, we select \(T\) such that, according to \eqref{key estimate}, 
	\[
	|||\mathscr{F}(\mathscr{V})||| = |||\mathscr{U}||| \leqslant C_4(T)T\varepsilon \leqslant \sqrt{\varepsilon}
	\quad \text{for all } \mathscr{V} \in X_{\varepsilon}.
	\]
	This condition is equivalent to
	\begin{equation}\label{eq:5.13}
		C_4(T) T \leqslant \frac{1}{\sqrt{\varepsilon}}.
	\end{equation}
		For \eqref{eq5.12}, let \(\mathscr{U}_1 = \mathscr{F}[\mathscr{V}_1]\), \(\mathscr{U}_2 = \mathscr{F}[\mathscr{V}_2]\), and \(\mathscr{W} = \mathscr{U}_1 - \mathscr{U}_2\). Then \(\mathscr{W}\) satisfies
		\begin{equation}\label{eq5.16}
			\left\{
			\begin{aligned}
				&\mathrm{i}\partial_t \mathscr{W}+\hat{\mathcal{L}}_2\mathscr{W}=  -\big(\hat{N}(\mathscr{V}_1)-\hat{N}(\mathscr{V}_2)\big) && \text{in } \mathbb{R}^d \times (0,T], \\
				&\mathscr{W}(\mathbf{x},0) = 0&& \text{in } \mathbb{R}^d.
			\end{aligned}
			\right.
		\end{equation} 
		Repeating the procedure to obtain estimate \eqref{eq5.10}, for problem \eqref{eq5.16}, we have, for any multiindex $\gamma$ with $0\leqslant|\gamma|\leqslant 1$,
		\begin{equation}\label{eq 5.17}
				\begin{aligned}
				&\quad\frac{d}{dt}||\partial^{\gamma}_{\mathrm{x}}\mathscr{W}(t)||^2_{L^{2}(\mathbb{R}^d)} 
				\\&\leqslant C\left(|| \partial^{\gamma}_{\mathrm{x}}\mathscr{W}(t)||^{2}_{L^{2}(\mathbb{R}^d)}+|| [\partial_\mathbf{x}^\gamma,\hat{\mathcal{L}}_2]\mathscr{W}(t)||^{2}_{L^{2}(\mathbb{R}^d)}+||\partial^{\gamma}_{\mathrm{x}}(\hat{N}(\mathscr{V}_1)-\hat{N}(\mathscr{V}_2))||^{2}_{L^{2}(\mathbb{R}^d)}\right)\\
				&\leqslant C \left(|| \mathscr{U}(t)||^{2}_{H^{1}(\mathbb{R}^d)}+||\partial^{\gamma}_{\mathrm{x}}(\hat{N}(\mathscr{V}_1)-\hat{N}(\mathscr{V}_2))||^{2}_{L^{2}(\mathbb{R}^d)}\right),
			\end{aligned}
		\end{equation}
			where $[\partial_\mathbf{x}^\gamma,\hat{\mathcal{L}}_2] =\partial_\mathbf{x}^\gamma \hat{\mathcal{L}}_2- \hat{\mathcal{L}}_2\partial_\mathbf{x}^\gamma $ denotes the commutator, 
		which is a differential operator of order 0 in spatial derivatives. By using \eqref{eq:hatN}, the fact that $\mathscr{V}_1,\mathscr{V}_2 \in X_{\varepsilon}$ implies $\|\mathscr{V}_i\|_{W^{1,\infty}(\mathbb{R}^d)} \leqslant C\sqrt{\varepsilon}$ for $i=1,2$ and Lemma \ref{Lemma 5.1}, we calculate the term $||\partial^{\gamma}_{\mathrm{x}}(\hat{N}(\mathscr{V}_1)-\hat{N}(\mathscr{V}_2))||^{2}_{L^{2}(\mathbb{R}^d)}$ as follows, 
		\begin{equation}\label{eq 5.18}
			\begin{aligned}
			&\quad\int_{\mathbb{R}^d}|\partial^{\gamma}_{\mathrm{x}}(\hat{N}(\mathscr{V}_1)-\hat{N}(\mathscr{V}_2))|^2d\mathrm{x}\\
			&\leqslant\int_{\mathbb{R}^d}2|\mathscr{V}_1-\mathscr{V}_2|^2|\partial^{\gamma}_{\mathbf{x}}\big(\sum_{k=2}^{l_0}\alpha_k\sum_{i=2}^{k}\binom{k}{i}u_0^{k-i}\sum_{j=0}^{i-1}\mathscr{V}^{i-1-j}_1\mathscr{V}^j_2\big)|^2\\
			& \quad +2|\partial^{\gamma}_{\mathbf{x}}\mathscr{V}_1-\partial^{\gamma}_{\mathbf{x}}\mathscr{V}_2|^2|\sum_{k=2}^{l_0}\alpha_k\sum_{i=2}^{k}\binom{k}{i}u_0^{k-i}\sum_{j=0}^{i-1}\mathscr{V}^{i-1-j}_1\mathscr{V}^j_2|^2d\mathbf{x}\\
			&\leqslant C ||\mathscr{V}_1-\mathscr{V}_2||^2_{H^{1}(D)}\sum_{k=2}^{l_0}\sum_{i=2}^{k}\varepsilon^{2k-2i}\varepsilon^{i-1}\leqslant C \varepsilon ||\mathscr{V}_1-\mathscr{V}_2||^2_{H^{1}(\mathbb{R}^d)}.
			\end{aligned}
		\end{equation}
		Using \eqref{eq 5.18} and summing over $\gamma$ with $0\leqslant|\gamma|\leqslant 1$ to \eqref{eq 5.17}, then applying Gronwall's inequality and using the definition of \(\|\cdot\|\) yields
		\begin{equation}\label{eq 5.19}
			||\mathscr{W}||^{2}\leqslant C_{5}(T)\varepsilon \int_{0}^{T}||\mathscr{V}_1(t)-\mathscr{V}_2(t)||^2_{H^{1}(\mathbb{R}^d)}dt \leqslant \varepsilon C_{5}(T)T ||\mathscr{V}_1-\mathscr{V}_2||^2
		\end{equation}
		where \(C_5(T)\) has the same dependencies as \(C_4(T)\) and is monotonically increasing in \(T\) with \(C_5(0) < \infty\). 	To ensure \(\|\mathscr{W}\|=\| \mathscr{F}[\mathscr{V}_1]-\mathscr{F}[\mathscr{V}_2]\| \leqslant \frac{1}{2} \|\mathscr{V}_1 - \mathscr{V}_2\|\), we require by \eqref{eq 5.19}, 
		\begin{equation}\label{eq C_T(5)}
			C_5(T) T \leqslant \frac{1}{4{\varepsilon}}.
		\end{equation}
		Set the lifespan time
	\begin{equation}\label{eq:T_lifespan}
		T_{\mathrm{lifespan}} := \max\{T \mid C_4(T)T \le \varepsilon^{-1/2}
		\ \text{and}\ C_5(T)T \le (4\varepsilon)^{-1}\}.
	\end{equation}
	For such $T$ the map $\mathscr{F}$ is a contraction on $X_\varepsilon$, and
	Banach's fixed point theorem gives a unique solution $\mathscr{U}\in X_\varepsilon$
	to \eqref{nonlinear hyperbolic equation with source term and zero initial condition}. Sobolev embedding provides the H\"older regularity.

		\medskip
		\textbf{Step 4: Uniform lower bounds for \(T\) of problem \eqref{nonlinear hyperbolic equation with source term and zero initial condition}.}
		Define	
		\begin{equation}\label{5.17}
			T_{\mathrm{lower}}:=\max\{T|	C_4(T) T \leqslant 1\ \text{and}\ C_5(T) T \leqslant \frac{1}{4}\}.
		\end{equation}
		Since \(\varepsilon < 1\), we have \(\frac{1}{\sqrt{\varepsilon}}>1\) and \(\frac{1}{4 \varepsilon} > \frac{1}{4}\). As \(C_4(T)T\) and \(C_5(T)T\) are continuous and increasing, it follows that \( T_{\mathrm{lifespan}}\geqslant T_{\mathrm{lower}} \), with \(T_{\mathrm{lower}}\) independent of \(\varepsilon\) which implies that properties 3 holds.
		
		The proof is complete.
	\end{proof}
	
	\begin{remark}\label{remark 5.3}
		The argument in the proof of Theorem~\ref{thm:5.2} shares the same spirit as that of Theorem 3 in Section 12.2 of \cite{Evans}. To ensure that the mapping $\mathscr{F}$ preserves the space $X_{\varepsilon}$ and satisfies the contraction property, the time $T$ is chosen to satisfy conditions \eqref{eq:5.13} and \eqref{eq C_T(5)}. Notably, as $\varepsilon$ decreases, the existence time $T$ for problem \eqref{1.2} increases, thereby generalizing the conclusion of Theorem 3 in Section 12.2 of \cite{Evans} in terms of the existence time $T$. In our setting, $T$ can be chosen arbitrarily large, whereas Theorem 3 in Section 12.2 of \cite{Evans} generally guarantees only a short-time existence. This enhancement stems from three key structural features: (i) the problem \eqref{nonlinear hyperbolic equation with source term and zero initial condition} has trivial initial data; (ii) the polynomial nonlinearity $\hat{N}$ has compact support in $D$ and an explicit form; and (iii) the source term $F_2$ has an explicit expression, with its $H^2(D)$-norm being small in terms of $\varepsilon$. Combined with the careful construction of the closed set $X_\varepsilon$ and the properties of the function $u_0$, these features enable the existence of a solution to problem \eqref{nonlinear hyperbolic equation with source term and zero initial condition} over an arbitrarily large time interval $[0, T]$.	\end{remark}
	
	We now proceed to prove Theorem \ref{thm:main2}. As the argument closely parallels that of Theorem \ref{1.1}, we provide the complete details below for clarity and self-containedness.
	\begin{proof}[Proof of Theorem \ref{thm:main2}.]\label{proof 1.2}
		In the three-dimensional case, according to Remark \ref{remark 3.8}, we select the predetermined parameters $r_0$, $\varepsilon$, and $m$ as follows:
		\begin{equation}\label{eeq 5.20}
			\begin{aligned}
				&\hfill r_0 = \min\left\{\frac{1}{3}(\mathcal{M} + 1)^{-1},\ \frac{1}{6}\min_{1 \leq i < j \leq n} |\mathbf{x}_i - \mathbf{x}_j|\right\}, \\
				&\hfill \varepsilon = \min\left\{\frac{1}{C_1 + (C_6)^2},\frac{1}{\left(C_4(T)T\right)^2},\frac{1}{4C_5(T)T},1\right\}, \\
				&\hfill m = \max\left\{ \left\lfloor \frac{256(2+C_2(r_0)\varepsilon)^2r_{0}^{3}-3}{2} \right\rfloor + 1,  M(r_0) \right\},
			\end{aligned}
		\end{equation}
		where the constants $C_1$, $C_2(r_0)$, and $M(r_0,\omega)$ are specified in Theorem~\ref{thm:main_result};  Constants $C_4(T)$ and $C_5(T)$ are given by \eqref{key estimate} and \eqref{eq 5.19}, respectively, in Theorem~\ref{thm:5.2};  Constant $C_6$ is given in Theorem~\ref{thm:5.2}; and the prescribed time $T$ is as given in Theorem~\ref{thm:main2}. Then, we choose the initial inputs for problem \eqref{1.2} as specified in \eqref{initial condition}, with the corresponding $u_0$ determined by the parameters $r_0$, $\varepsilon$, and $m$ in \eqref{eeq 5.20}. Let $\mathscr{U}$ denote the unique solution to auxiliary problem \eqref{nonlinear hyperbolic equation with source term and zero initial condition}. A direct computation verifies that
		\begin{equation}\label{eq 5.21}
			U := u_0 + \mathscr{U}
		\end{equation}
		constitutes the unique solution to problem \eqref{1.2}. Since $u_0$ is smooth and defined on $[0,\infty)$, the lifespan of $U$ is entirely determined by that of $\mathscr{U}$. As established by the lifespan of $\mathscr{U}$ in \eqref{eq:T_lifespan} and the explicit expression for $\varepsilon$ in \eqref{eeq 5.20}, $U$ exists on the interval $[0,T]$, where $T$ denotes the prespecified time  given in Theorem \ref{thm:main2}. Moreover, Theorems \ref{thm:5.2} and \ref{thm:main_result} together with Remark \ref{remark 2.3} yield $U \in L^{\infty}\big(0,T;H^{3}_{\text{loc}}(\mathbb{R}^d)\big)$.
		
		Next, applying the mean value theorem to \eqref{u0 smalles} and \eqref{3 d} in Theorem \ref{3.1} yields points $\mathbf{z}_i \in B_{3r_0}(\mathbf{x}_i) \backslash \overline{D}$ ($i=1,2,\dots,n$) such that for all $t \in [0,T]$:
		\begin{equation}\label{5.23}
			|\nabla u_0(\mathbf{z}_i,t)| \geqslant \frac{
				\tfrac{1}{16}\sqrt{2m+3} \, r_0^{-3/2} - \left(C_1 + C_2(r_0)\right)\varepsilon
			}{3r_0}.
		\end{equation}
		Combining \eqref{eeq 5.20}, \eqref{eq 5.21}, \eqref{5.23}, and properties (3) in Theorem \ref{thm:5.2}, we obtain
		\begin{equation}\label{gradient estiamte for U near zi}
			\begin{aligned}
				|\nabla U(\mathbf{z}_i,t)| &\geqslant \frac{\left(
					\tfrac{1}{16}\sqrt{2m+3} \, r_0^{-3/2} -C_2(r_0)\varepsilon\right)-  C_1\varepsilon
				}{3r_0} - C_6\sqrt{\varepsilon} \\
				& > \frac{2 - 1}{3r_0} - 1 = \frac{1}{3r_0} - 1 > \mathcal{M},
			\end{aligned}
		\end{equation}
		for a.e. $t \in [0,T]$. Here, we choose $m \geqslant \left\lfloor \frac{256(2 + C_2(r_0)\varepsilon)^2 r_{0}^{3} - 3}{2} \right\rfloor +1$ in \eqref{eeq 5.20} to guarantee that  $	\tfrac{1}{16}\sqrt{2m+3} \, r_0^{-3/2} -C_2(r_0)\varepsilon>2 $. On the other hand, using above estimate and properties (3) in Theorem \ref{thm:main_result}, we calculate  
		\begin{equation}\label{quotient of gradient U}
		\frac{|\nabla U(\mathbf{z}_i,t)|}{||U(\cdot,t)||_{C^{1,\frac{1}{2}}(\overline{D})}}\geqslant \frac{\mathcal{M}}{||u_0(\cdot,t)||_{C^{1,\frac{1}{2}}(\overline{D})}+||\mathscr{U}(\cdot,t)||_{C^{1,\frac{1}{2}}(\overline{D})}} \geqslant \frac{\mathcal{M}}{C_1\varepsilon+C_6\sqrt{\varepsilon}} > \frac{\mathcal{M}}{2}
		\end{equation}
		for a.e. $t\in [0,T]$. 	By \eqref{gradient estiamte for U near zi} and \eqref{quotient of gradient U}, properties (1) and (2) in Theorem \ref{thm:main2} are verified. 
		
		Furthermore, under the above construction, the localized gradient amplification points are confined within the balls $B_{3r_0}(\mathbf{x}_i)$ which implies that 
		\begin{equation*}
			\operatorname{meas}\left\{ \mathbf{x} \in \bigcup_{i=1}^{n} \left(B_{3r_0}(\mathbf{x}_i)\backslash \overline{D}\right) : |\nabla U(\mathbf{x},t)| > \mathcal{M} \right\} < 36n\pi r_0^3 < \frac{100n}{(\mathcal{M} + 1)^3}  \to 0
		\end{equation*}
		as $\mathcal{M} \to \infty$, for almost every $t \in [0,T]$. Properties (3) in Theorem \ref{thm:main2} is verified. The two-dimensional case follows analogously.
		
		The proof is complete.
	\end{proof}

	\section*{Acknowledgment}
	The work of H. Diao is is supported by National Natural Science Foundation of China  (No. 12371422), the Fundamental Research Funds for the Central Universities, JLU.   The work of H. Liu is supported by the Hong Kong RGC General Research Funds (projects 11304224, 11311122 and 11303125),  the NSFC/RGC Joint Research Fund (project N\_CityU101/21), the France-Hong Kong ANR/RGC Joint Research Grant, A-CityU203/19.

	\medskip
	
	%\Blu{\textbf{Data availability statement.} Data sharing is not applicable to this article as no datasets were generated or analysed during the current study.}


\begin{thebibliography}{99}
		
		\bibitem{Abramowitz}
		M.~Abramowitz and I.~A.~Stegun,
		\textit{Handbook of Mathematical Functions with Formulas, Graphs, and Mathematical Tables},
		Dover Publications, New York, 1965.
		
		\bibitem{Anderson1958}
		P.~W.~Anderson,
		Absence of diffusion in certain random lattices,
		\textit{Phys. Rev.},
		\textbf{109} (1958), 1492--1505.
		
		\bibitem{Antonini}
		C.~A.~Antonini,
		Smooth approximation of Lipschitz domains, weak curvatures and isocapacitary estimates,
		\textit{Calc. Var. Partial Differential Equations},
		\textbf{63} (2024), Paper No.~91, 34 pp.
		
		\bibitem{Cazenave2003Semilinear}
		T.~Cazenave,
		\textit{Semilinear Schr\"odinger Equations},
		Courant Lecture Notes in Mathematics, vol.~10,
		New York University, Courant Institute of Mathematical Sciences, New York;
		American Mathematical Society, Providence, RI, 2003.
		
		\bibitem{ColtonKress}
		D.~Colton and R.~Kress,
		\textit{Inverse Acoustic and Electromagnetic Scattering Theory},
		4th ed.,
		Applied Mathematical Sciences, vol.~93,
		Springer, Cham, 2019.
		
		\bibitem{Cycon1987}
		H.~L.~Cycon, R.~G.~Froese, W.~Kirsch, and B.~Simon,
		\textit{Schr\"odinger Operators: With Applications to Quantum Mechanics and Global Geometry},
		Theoretical and Mathematical Physics,
		Springer-Verlag, Berlin, 1987.
		
		\bibitem{Evans}
		L.~C.~Evans,
		\textit{Partial Differential Equations},
		2nd ed.,
		Graduate Studies in Mathematics, vol.~19,
		American Mathematical Society, Providence, RI, 2010.
		
		\bibitem{Fan2017LogLogMultiPoint}
		C.~Fan,
		Log--log blow-up solutions blowing up at exactly \(m\) points,
		\textit{Ann. Inst. H. Poincar\'e Anal. Non Lin\'eaire},
		\textbf{34} (2017), no.~5, 1429--1482.
		
		\bibitem{GilbargTrudinger}
		D.~Gilbarg and N.~S.~Trudinger,
		\textit{Elliptic Partial Differential Equations of Second Order},
		2nd ed.,
		Classics in Mathematics,
		Springer-Verlag, Berlin, 2001.
		
		% \bibitem{GinibreVelo1979NLS}
		% J.~Ginibre and G.~Velo,
		% On a class of nonlinear Schr\"odinger equations. I. The Cauchy problem, general case,
		% \textit{J. Funct. Anal.},
		% \textbf{32} (1979), 1--32.
		
		\bibitem{Glassey1977Blowup}
		R.~T.~Glassey,
		On the blowing up of solutions to the Cauchy problem for nonlinear Schr\"odinger equations,
		\textit{J. Math. Phys.},
		\textbf{18} (1977), 1794--1797.
		
		\bibitem{Heisenberg1927}
		W.~Heisenberg,
		\"Uber den anschaulichen Inhalt der quantentheoretischen Kinematik und Mechanik,
		\textit{Z. Phys.},
		\textbf{43} (1927), 172--198.
		
		\bibitem{LionsMagenes1972NHBVP}
		J.~L.~Lions and E.~Magenes,
		\textit{Non-Homogeneous Boundary Value Problems and Applications. Vol.~I},
		translated from the French by P.~Kenneth,
		Die Grundlehren der mathematischen Wissenschaften, Band 181,
		Springer-Verlag, New York--Heidelberg, 1972.
		
		\bibitem{LiuZou}
		H.~Liu and J.~Zou,
		Zeros of the Bessel and spherical Bessel functions and their applications for uniqueness in inverse acoustic obstacle scattering,
		\textit{IMA J. Appl. Math.},
		\textbf{72} (2007), no.~6, 817--831.
		
		\bibitem{Lohofer}
		G.~Loh\"ofer,
		Inequalities for the associated Legendre functions,
		\textit{J. Approx. Theory},
		\textbf{95} (1998), no.~2, 178--193.
		
		\bibitem{Merle1990kBlowup}
		F.~Merle,
		Construction of solutions with exactly \(k\) blow-up points for the Schr\"odinger equation with critical nonlinearity,
		\textit{Comm. Math. Phys.},
		\textbf{129} (1990), no.~2, 223--240.
		
		\bibitem{ReedSimon3}
		M.~Reed and B.~Simon,
		\textit{Methods of Modern Mathematical Physics, III: Scattering Theory},
		Academic Press, New York, 1979.
		
		\bibitem{Sogge}
		C.~D.~Sogge,
		\textit{Lectures on Nonlinear Wave Equations},
		Monographs in Analysis, II,
		International Press, Boston, MA, 1995.
		
		\bibitem{SulemSulemNLS}
		C.~Sulem and P.-L.~Sulem,
		\textit{The Nonlinear Schr\"odinger Equation: Self-Focusing and Wave Collapse},
		Applied Mathematical Sciences, vol.~139,
		Springer-Verlag, New York, 1999.
		
		\bibitem{Tao2006NDE}
		T.~Tao,
		\textit{Nonlinear Dispersive Equations: Local and Global Analysis},
		CBMS Regional Conference Series in Mathematics, vol.~106,
		American Mathematical Society, Providence, RI, 2006.
		
		\bibitem{Taylor1972}
		J.~R.~Taylor,
		\textit{Scattering Theory: The Quantum Theory of Nonrelativistic Collisions},
		John Wiley \& Sons, New York, 1972.
		
		\bibitem{Tsutsumi1984DampedNLS}
		M.~Tsutsumi,
		Nonexistence of global solutions to the Cauchy problem for the damped nonlinear Schr\"odinger equations,
		\textit{SIAM J. Math. Anal.},
		\textbf{15} (1984), no.~2, 357--366.
		
		\bibitem{Weck}
		N.~Weck,
		Approximation by Herglotz wave functions,
		\textit{Math. Methods Appl. Sci.},
		\textbf{27} (2004), no.~2, 155--162.
		
		\bibitem{Yafaev1992}
		D.~R.~Yafaev,
		\textit{Mathematical Scattering Theory: General Theory},
		Translations of Mathematical Monographs, vol.~105,
		American Mathematical Society, Providence, RI, 1992.
		
	\end{thebibliography}
\end{document}